\documentclass[reqno]{amsart}
\usepackage{amsfonts}
\usepackage{amssymb}
\usepackage{color}

\theoremstyle{plain}
\newtheorem{thm}{Theorem}
\newtheorem{cor}{Corollary}
\newtheorem{lem}{Lemma}
\newtheorem{proposition}{Proposition}

\newtheorem{exa}{Example}

\theoremstyle{definition}
\newtheorem{defn}{Definition}
\theoremstyle{remark}
\newtheorem{rem}{Remark}
\numberwithin{equation}{section}

\makeatletter

\newcommand{\Rmnum}[1]{\expandafter\@slowromancap\romannumeral #1@}
\makeatother

\begin{document}
\title{Fractal dimensions and quasisymmetric packing-minimalities of some homogeneous Moran sets}
\author{Pingping Liu}
\address[Pingping Liu]{School  of Mathematics, Guangxi University, Nanning, 530004, P.~R. China}
\email{2306301021@st.gxu.edu.cn}

\author{Chenyuan JIANG}
\address[Chenyuan JIANG]{School  of Mathematics, Guangxi University, Nanning, 530004, P.~R. China}
\email{2306301010@st.gxu.edu.cn}

\author{Yanzhe Li}
\address[Yanzhe Li]{School of Mathematics, Guangxi University; Guangxi Center for Mathematical Research; Center for Applied Mathematics of Guangxi(Guangxi University), Nanning, 530004, P.~R. China}
\email{liyz@gxu.edu.cn}

\thanks{This work is supported by National Natural Science Foundation of China (No.12461015) and Guangxi Natural Science Foundation (2020GXNSFAA297040). }
\thanks{Yanzhe Li is the corresponding author. }
\thanks{These authors contributed equally to this work.}

\begin{abstract}
In this paper, we study the upper box dimensions, the packing dimensions and the quasisymmetric  minimalities of some homogeneous Moran sets. We obtain the upper box dimension formula and the packing dimension formula for three classes of homogeneous Moran sets under some conditions. And two classes of them with packing dimension 1 are shown to be quasisymmetrically  packing-minimal.
\end{abstract}

\maketitle

\section{Introduction}
\label{sec:intro}

The homogeneous Moran sets is an important class of  fractal sets in the study of the fractal geometry and related fields, and the fractal dimensions are useful tools in describing the complexities and irregularities of fractal sets, while the box dimension(the upper box dimension and the lower box dimension are denoted by $\overline{\dim}_{B}$ and $\underline{\dim}_{B}$) and the packing dimension(the packing dimension is denoted by $\dim_{P}$) are two classical fractal dimensions.
There are some classical results about the box dimensions and the packing dimensions of the homogeneous Moran sets. In \cite{fdj97}, Feng, Wen and Wu studied some properties of the homogeneous Moran sets and obtained the range of the upper box dimensions and the packing dimensions of the family of homogeneous Moran sets. In \cite{wzywuj05}, Wen and Wu constructed a class of special homogeneous Moran sets by the gaps between the basic intervals, which is called the homogeneous perfect sets. Later, Wang and Wu obtained the packing and upper box dimensions of the homogeneous perfect sets in \cite{wxywuj08} under some conditions.

In Theorem \ref{thm1} of this paper, we obtain the packing and upper box dimensions of three classes of homogeneous Moran sets  under some conditions, which generalizes the result in \cite{wxywuj08}.

 Let $f: \mathbb{R}^n\to \mathbb{R}^n$ be a homeomorphism, if there exists a homeomorphism $\eta:[0,\infty)\to \mathbb{R}^n$, such that for any triples $a, b, x$ of distinct points in $\mathbb{R}^n$,
\begin{equation*}
  \frac{|f(x)-f(a)|}{|f(x)-f(b)|}\le\eta(\frac{|x-a|}{|x-b|}),
\end{equation*}
then $f$ is called a $n-$dimensional quasisymmetric mapping.

 The quasisymmetric mappings contain some well-known mappings, such as the Lipschitz mappings. However, the fractal dimensions of the fractal sets may be changed under the quasisymmetric mappings, where the Lipschitz mappings preserve the fractal dimensions. A set $E\subset\mathbb{R}^{n}$ is called a quasisymmetrically Hausdorff minimal set(or $E$ is called quasisymmetrically Hausdorff-minimal) if $\dim_{H}f(E)\ge \dim_{H}E$ for all $n-$dimensional quasisymmetric mapping $f$, where $\dim_{H}E$ denotes the Hausdorff dimension of $E$. Similarly, we can define the quasisymmetrically packing minimal set by packing dimension.

There are many classical results on quasisymmetric Hausdorff minimality, see in \cite{gwvj73,gw73,bj99,tj00,kv06,hh06,dyx11,ww14,yjj18}. However, since it is always difficult to estimate the packing dimension, the results on the quasisymmetric minimality on packing dimension are less than those on the Hausdorff dimension.
Kovalev showed in \cite{kv06} that the packing dimension of a quasisymmetrically packing minimal set in $\mathbb{R}$ is either 0 or 1. It is easy to prove that any set with packing dimension 0 is quasisymmetrically packing-minimal(see in \cite{Ahl06}), then we focous on the quasisymmetric packing minimality on the sets in $\mathbb{R}$ with packing dimension 1. For the fractal sets   in $\mathbb{R}$ with packing dimension 1,
Li, Wu and Xi found that two  classes of Moran sets  are quasisymmetrically packing-minimal, see in \cite{lwx13}.
Wang and Wen showed in \cite{ww14} that all uniform Cantor sets  are quasisymmetrically packing-minimal.
In \cite{yjj18}, Yang, Wu and Li generalized the result in \cite{ww14} and proved that some homogeneous perfect sets are quasisymmetrically packing-minimal.
In \cite{lly25}, Liu, Li and Yang generalized the result in \cite{yjj18} and proved that a large class of homogeneous perfect sets is quasisymmetrically packing-minimal.

In Theorem \ref{thm2} of this paper, we use the  packing dimension  formula  of the homogeneous Moran sets in Theorem \ref{thm1} to  prove that two classes of homogeneous Moran sets with packing dimension 1 are quasisymmetrically packing-minimal, our result generalizes the result in \cite{lly25}.

This paper is organized as follows. In Section 2, we introduce the notion of the homogeneous Moran sets and give some notations. Our main results are stated in Section 3. The proofs of our main results are given in Section 4, Section 5 and Section 6. Finally, we give some examples in Section 7 to  show that our results generalize the previous results and our conclusions may not hold if some conditions are lacked.

\bigskip

\section{Preliminaries}
\label{sec:pre}
In this section, we introduce the notion of the homogeneous Moran sets and give some
notations which will be used in further discussions.

Let $\{n_k\}_{k\ge 1}$ and $\{c_k\}_{k\ge1}$ be sequences of positive integers and  positive real numbers respectively, satisfying  $n_k\ge 2$ and $n_kc_k<1$ for any $k\ge 1$.  For any $k\ge1$, let $D_{k}=\{i_{1}i_{2}\cdots i_{k}:1\le i_{j} \le n_{j},1\le j \le k\}$, $D_{0}=\emptyset$ and $D=\cup_{k\ge0}D_{k}$. If $\sigma=\sigma_{1}\sigma_{2}\cdots\sigma_{k}\in D_{k}$, $\tau=\tau_{1}\tau_{2}\cdots\tau_{m}$, where $1\le \tau_i\le n_{k+i}$ for any $1\le i \le m$, then $\sigma*\tau =\sigma_{1}\sigma_{2}\cdots\sigma_{k}\tau_{1}\tau_{2}\cdots\tau_{m}\in D_{k+m}$.

\begin{defn}\label{HMS}(Homogeneous Moran sets \rm{\cite{hua00}})
Let $I_0=[0,1]$, and let $\mathcal I=\{I_\sigma:\sigma\in D \}$ be the collection of some closed subintervals of $I_0$. The collection $\mathcal I$ is said to have homogeneous Moran structure if it satisfies:
\begin{enumerate}
\item[\textup{(1)}] $I_{\emptyset}=I_{0}$;
\item[\textup{(2)}] For any $k\ge1$, $\sigma\in{D_{k-1}}$ and $1\le l \le n_{k}$, $I_{\sigma * l}$ is a closed subinterval of $I_{\sigma}$  satisfying $\min(I_{\sigma*(l+1)})\ge \max(I_{\sigma*l})$ for any $1\le l \le n_{k}-1$, which implies that the interiors of $I_{\sigma*i}$ and $I_{\sigma*j}$ are disjoint for any $1\le i<j \le n_{k}$;
\item[\textup{(3)}] For any $k\ge1$ and $\sigma\in{D_{k-1}}$, $1\le i \le j \le n_{k}$, we have $$\frac{\left|I_{\sigma * i}\right|}{\left|I_{\sigma}\right|}=\frac{\left|I_{\sigma * j}\right|}{\left|I_{\sigma}\right|}=c_{k},$$
where the symbol $\left|A\right|$ is used to denote the diameter of $A\subset \mathbb{R}$ .
\end{enumerate}

If $\mathcal{I}$ has homogeneous Moran structure, define $E_{k}=\cup_{\sigma\in{D_{k}}}I_{\sigma}$ for any $k\ge0$,  then  we call the  nonempty compact set
    $E=\cap_{k\ge0}E_{k}=E(I_{0},\{n_{k}\},\{c_{k}\})$  a homogeneous Moran set.
    For each $k\ge0$, define $\mathcal{I}_{k}=\left\{I_{\sigma}:\sigma\in{D_{k}}\right\}$, then every element $I_{\sigma}\in \mathcal{I}_{k}$ is referred to as a  $k$-order basic interval of $E$. The class of all homogeneous Moran sets generated by $I_{0},\{n_{k}\},\{c_{k}\}$ is denoted by $\mathcal{M}(I_{0},\{n_{k}\},\{c_{k}\})$.
\end{defn}

\bigskip
Next, we introduce some notation for later use. For any $k\ge 1$ and $\sigma \in D_{k-1}$,
$1\le i \le n_k-1$, write
\begin{equation}
 \begin{aligned}
     \xi_{\sigma,0}&=\min(I_{\sigma*1})-\min(I_{\sigma}),\\
     \xi_{\sigma,i}&=\min(I_{\sigma*(i+1)})-\max(I_{\sigma*i}),\\
     \xi_{\sigma,n_k}&=\max(I_{\sigma})-\max(I_{\sigma*n_k}),\nonumber
 \end{aligned}
 \end{equation}
then the collection $\{\xi_{\sigma,l}:\sigma \in D_{k-1}, 0\le l \le n_k\}$ consists of nonnegative real numbers for any $k\ge 1$. For any  $\sigma \in D_{k-1}, 1\le l \le n_k-1$, $\xi_{\sigma,l}$ is referred to as a  length of a $k$-order gap of $E$.

For any $k\ge 1$, let $\overline\nu_k$ and $\underline\nu_k$ denote the maximum and minimum values of the lengths of the $k$-order gaps of $E$ respectively, then
$$\overline\nu_k=\max\limits_{\sigma\in D_{k-1}, 1\leq i\leq n_k-1}\xi_{\sigma,i},\quad \underline\nu_k=\min\limits_{\sigma\in D_{k-1}, 1\leq i\leq n_k-1}\xi_{\sigma,i}.$$
Denote by $N_k$ the count of the $k$-order basic intervals of $E$ and  denote by $\delta_k$ the length of each $k$-order basic interval of $E$, which means
\begin{align}
\nonumber
N_k=\prod_{i=1}^{k} n_i,\quad \delta_k=\prod_{i=1}^{k} c_i.
\end{align}
Write $l(E_k)$ for the sum of the lengths of all $k$-order basic intervals of $E$, thus $l(E_k)=N_k\delta_k$.

\begin{rem}\label{rm2} For any  $k\ge 1$, $\sigma_1 \in D_{k-1}$, $\sigma_2 \in D_{k-1}$, $\sigma_1\ne \sigma_2$  and $0\le l \le n_k$, $\xi_{\sigma_1,l}$  is not necessarily equal to $\xi_{\sigma_2,l}$. However, if $E=E(I_0,\{n_{k}\}, \{c_{k}\},\{\xi_{k,l}\})$ is a homogeneous perfect set(the definition of the homogeneous perfect sets can be found in \cite{wzywuj05}, \cite{wxywuj08}), then for any  $k\ge 1$, $\sigma_1 \in D_{k-1}$, $\sigma_2 \in D_{k-1}$, $\sigma_1\ne \sigma_2$  and $0\le l \le n_k$, we have $\xi_{\sigma_1,l}=\xi_{\sigma_2,l}=\xi_{k,l}$.
\end{rem}

\bigskip

\bigskip

\section{Main results}
The main results of this paper are stated in the following theorems.

\begin{thm}\label{thm1}
     Suppose
    $E\in \mathcal{M}(I_{0},\left\{n_{k}\right\},\left\{c_{k}\right\})$ satisfying the following condition:
there are two nonnegative real sequences $\{L_k\}_{k\ge1}$ and $\{R_k\}_{k\ge 1}$, such that
\begin{equation}
 \begin{aligned}
    \xi_{\sigma,0}=L_{k+1},\quad
    \xi_{\sigma,n_{k+1}}=R_{k+1} \nonumber
 \end{aligned}
 \end{equation}
 for any $k\ge 0$, $\sigma\in D_{k}$.

And assume that for any $k\ge 1$, at least one of the three conditions below holds:
\begin{itemize}
  \item[(A)]there is $w_1>0$, such that $\overline\nu_k\le w_1 \delta_k$;
  \item[(B)]there is $w_2>0$, such that $\overline\nu_k\le w_2\underline\nu_k$;
  \item[(C)]there is $w_3>0$, such that $n_k\underline\nu_k\ge w_3\delta_{k-1}$.
\end{itemize}

Then
\begin{equation}\label{031}
\dim_PE=\overline\dim_BE=\limsup\limits_{k \to\infty}  \frac{\log n_1n_2\cdots n_kn_{k+1}}{-\log(\delta_k-L_{k+1}-R_{k+1})+\log n_{k+1}}.
\end{equation}
\end{thm}

\begin{rem}\label{rm4}
 Any homogeneous perfect set $E=E(I_{0},\{n_{k}\}, \{c_{k}\},\{\xi_{k,l}\})$ is a homogeneous Moran set with $\xi_{\sigma,0}=\xi_{k,0}=L_{k}$, $\xi_{\sigma,n_k}=\xi_{k,n_k}=R_{k}$, and $\xi_{\sigma,l}=\xi_{k,l}$ for any $k\ge 1$, $\sigma\in D_{k-1}$ and $1\le l\le n_k-1$ (the definition of homogeneous perfect sets can be found in \cite{wzywuj05},\cite{wxywuj08}). Obviously, the homogeneous Moran sets satisfying condition (A) or (B) or (C) of Theorem \ref{thm1} of this paper contain the homogeneous perfect sets satisfying condition (A) or (C) or (B) of Theorem 1.4 of \cite{wxywuj08}. Conversely, according to Remark \ref{rm2} of this paper, the homogeneous Moran set satisfying the conditions of Theorem \ref{thm1} of this paper may not be a homogeneous perfect set. It should be noted that equation (1.4) of Theorem 1.4 of \cite{wxywuj08} and equation (\ref{031}) of this paper are equivalently. Therefore, Theorem \ref{thm1} of this paper is a  generalization of Theorem 1.4 of \cite{wxywuj08}. In addition, we give an example in  Section 7 of this paper to show the above meaning in detail.
\end{rem}

\bigskip

\begin{thm}\label{thm2}
	Suppose
    $E\in \mathcal{M}(I_{0},\left\{n_{k}\right\},\left\{c_{k}\right\})$ satisfying the following condition:
there are two nonnegative real sequences $\{L_k\}_{k\ge1}$ and $\{R_k\}_{k\ge 1}$, such that
\begin{equation}
 \begin{aligned}
    \xi_{\sigma,0}=L_{k+1},\quad
    \xi_{\sigma,n_{k+1}}=R_{k+1} \nonumber
 \end{aligned}
 \end{equation}
 for any $k\ge 0$, $\sigma\in D_{k}$.

 And assume that for any $k\ge 1$, at least one of the two conditions below holds:
\begin{itemize}
  \item[(A)]there is $w_1>0$, such that $\overline\nu_k\le w_1\delta_k $;
  \item[(B)]there is $w_2>0$, such that $\overline\nu_k\le w_2\underline\nu_k$.
\end{itemize}

If $\dim_{P}E=1$, then for any \text{\rm1}-dimensional quasisymmetric mapping $f$, we obtain $\operatorname{dim}_{P}f(E)=1$,  which yields  $E$ is a quasisymmetrically  packing-minimal set.

\end{thm}

\begin{rem}\label{rm5}
Repeating the analysis of Remark \ref{rm4},  we have that any homogeneous perfect set $E=E(I_{0},\{n_{k}\}, \{c_{k}\},\{\xi_{k,l}\})$  satisfying  the conditions in Theorem 1 of \cite{lly25} is a homogeneous Moran set with $E\in \mathcal{M}(I_{0},\left\{n_{k}\right\},\left\{c_{k}\right\})$  satisfying  condition (B) of Theorem \ref{thm2} of this paper, and   the homogeneous Moran set satisfying the conditions of Theorem \ref{thm2} of this paper may not be a homogeneous perfect set. Therefore, Theorem \ref{thm2} of this paper is a  generalization of Theorem 1 of \cite{lly25}.  Besides, we also give an example  in  Section 7 of this paper to show the above meaning in detail.

\end{rem}
\bigskip

\bigskip\section{The reconstruction of the Homogeneous Moran sets}

In order to prove the Theorem \ref{thm1} and Theorem \ref{thm2} conveniently, we do the equivalent reconstruction to the homogeneous Moran set $E\in \mathcal{M}(I_{0},\left\{n_{k}\right\},\left\{c_{k}\right\})$ satisfying the conditions of Theorem \ref{thm1} and Theorem \ref{thm2}.

For any $k\ge 0$, $\sigma\in D_{k}$, define $I_{\sigma}^{*}$  to be a closed interval contained in $I_{\sigma}$ fulfilling the conditions below:
\begin{enumerate}
  \item[(a)]$\min(I^{*}_\sigma)-\min(I_\sigma)=\xi_{\sigma,0}=L_{k+1},\quad
            \max(I_\sigma)-\max(I^{*}_\sigma)=\xi_{\sigma,n_{k+1}}=R_{k+1}$;
  \item[(b)]$\left | I^{*}_ {\sigma}\right | =\sum_{j=1}^{n_{k+1}-1} \xi_{\sigma,j}+n_{k+1}\delta_{k+1}=\delta_k-L_{k+1}-R_{k+1}$.
\end{enumerate}

 Let $I_{0}^{*}=I_{\emptyset}^{*}$, write $\delta^*_{0}=\left|I_{0}^{*}\right|$, $\delta^*_{k}=\left|I_{\sigma}^{*}\right|$ for any $ k\ge 1$ and $\sigma\in D_k$, then
\begin{equation}\label{041}
\delta^*_{k}=\delta_{k}-L_{k+1}-R_{k+1}.
\end{equation}
Suppose that $E_{k}^{*}=\cup_{\sigma\in{D_{k}}}I_{\sigma}^{*}$ for any $k\ge 0$, then
\begin{equation*}
E=\bigcap_{k\ge 0}\bigcup_{\sigma \in D_k}I_{\sigma}=\bigcap_{k\ge 0}\bigcup_{\sigma \in D_k}I^{*}_{\sigma}=\bigcap_{k\ge 0}E_{k}^{*}.
\end{equation*}
For any $ k\ge 1$ and $\sigma\in D_k$, we refer to $I^*_\sigma$ as a $k$-order first reconstructed basic interval of $E$.

In fact, for any $k\ge 0$ and $\sigma \in D_k$, $E\in \mathcal{M}(I_{0}^{*},\left\{n_{k}^{*}\right\},\left\{c_{k}^{*}\right\})$ is a homogeneous Moran set with the parameters given as follows:
\begin{enumerate}    \item[\textup{(1)}]$I_{0}^{*}=I_{0}-[\min(I_{0}), \min(I_{0})+\xi_0)-(\max(I_{0})-\xi_{n_1}, \max(I_{0})]$;
    \item[\textup{(2)}]$c_{k+1}^{*}=\frac{\delta^*_{k+1}}{\delta^*_{k}}$, $n_{k+1}^{*}=n_{k+1}$.
\end{enumerate}

 For any $k\ge 1$ and $\sigma \in D_{k-1}$,
$1\le i \le n_k-1$, write
\begin{equation}
 \begin{aligned}
     \xi^{*}_{\sigma,0}&=\min(I^{*}_{\sigma*1})-\min(I^{*}_{\sigma});\\
     \xi^{*}_{\sigma,i}&=\min(I^{*}_{\sigma*(i+1)})-\max(I^{*}_{\sigma*i});\\
     \xi^{*}_{\sigma,n_k}&=\max(I^{*}_{\sigma})-\max(I^{*}_{\sigma*n_k}).\nonumber
 \end{aligned}
 \end{equation}
Then for any $k\ge 1$,  the collection $\{\xi^{*}_{\sigma,l}:\sigma \in D_{k-1}, 0\le l \le n_k\}$ consists of nonnegative real numbers,  and for any  $\sigma \in D_{k-1}, 1\le l \le n_k-1$,  $\xi^{*}_{\sigma,l}$ is referred to as a length of a $k$-order first reconstructed gap of $E$.

For any $k\ge 0$, $\sigma\in D_{k}$ and $1\le l \le n_{k+1}-1$, we obtain
\begin{align}\label{042}
\xi^*_{\sigma,l}=\xi_{\sigma,l}+\xi_{\sigma*l,n_{k+2}}+\xi_{\sigma*(l+1),0}=\xi_{\sigma,l}+R_{k+2}+L_{k+2},
\end{align}
\begin{align}\label{043}
 \xi^*_{\sigma,0}=\xi_{\sigma*1,0}=L_{k+2}, \quad
 \xi^*_{\sigma,n_{k+1}}=\xi_{\sigma*{n_{k+1}},n_{k+2}}=R_{k+2},
\end{align}
which implies
\begin{align}\label{044}
\sum_{l=1}^{n_{k+1}-1}\xi_{\sigma,l}^*=\sum_{l=1}^{n_{k+1}-1}(\xi_{\sigma,l}+R_{k+2}+L_{k+2})=\sum_{l=1}^{n_{k+1}-1}\xi_{\sigma,l}+(n_{k+1}-1)(R_{k+2}+L_{k+2}).
\end{align}

For any $k\ge 0$,  let $\overline\nu^*_{k+1}$ and $\underline\nu^*_{k+1}$ denote the maximum and minimum values of the lengths of the $(k+1)$-order  first reconstructed  gaps of $E$ respectively, then
$$\overline\nu^*_{k+1}=\max\limits_{\sigma\in D_{k}, 1\leq i\leq n_{k+1}-1}\xi^*_{\sigma,i},\quad
\underline\nu^*_{k+1}=\min\limits_{\sigma\in D_{k}, 1\leq i\leq n_{k+1}-1}\xi^*_{\sigma,i},$$
then by (\ref{042}),  we  have
\begin{equation}\label{045}
\overline\nu^*_{k+1}=\overline\nu_{k+1}+L_{k+2}+R_{k+2},
\end{equation}
\begin{equation}\label{046}
\underline\nu^*_{k+1}=\underline\nu_{k+1}+L_{k+2}+R_{k+2}.
\end{equation}
It is obvious that
\begin{equation}\label{047}
\underline\nu_{k+1}\le \underline\nu^*_{k+1},\quad  \overline\nu_{k+1}\le\overline\nu^*_{k+1}.
\end{equation}
By (\ref{043}), (\ref{045}) and (\ref{046}), we have
\begin{equation}\label{048}
\xi^*_{\sigma,0}+\xi^*_{\sigma,n_{k+1}}=L_{k+2}+R_{k+2}\le \underline\nu^*_{k+1}\le \overline\nu^*_{k+1}.
\end{equation}

For any $k\ge 1$,  denote by $N^*_{k}$ the count of the $k$-order first reconstructed basic intervals of $E$, and denote by $\delta^*_{k}$ the length of each $k$-order first reconstructed basic interval of $E$, and write $l(E_{k}^*)$ for the sum of the lengths of all $k$-order first reconstructed basic intervals of $E$, then
\begin{equation}\label{049}
N^*_{k}=\prod_{i=1}^{k} n_{i}^*=\prod_{i=1}^{k} n_{i}, \quad \delta_{k}^*=\delta_{0}^*\prod_{i=1}^{k} c_{i}^*,\quad l(E_{k}^*)=N^*_{k}\delta_{k}^*.
\end{equation}

Notice that $n_{k}^*=n_k$ for any $k\ge 1$, then by (\ref{041}) we obtain
\begin{equation*}
\limsup\limits_{k \to\infty}  \frac{\log n_1n_2\cdots n_kn_{k+1}}{-\log(\delta_k-L_{k+1}-R_{k+1})+\log n_{k+1}}=\limsup\limits_{k \to\infty}  \frac{\log n_{1}^*n_{2}^*\cdots n_{k}^*n_{k+1}^*}{-\log\delta_{k}^*+\log n_{k+1}^*}.
\end{equation*}
For convenience, we write
\begin{equation}\label{t}
t=\limsup\limits_{k \to\infty}  \frac{\log n_1n_2\cdots n_kn_{k+1}}{-\log\delta_{k}^*+\log n_{k+1}},
\end{equation}
then  to get (\ref{031}), it is only need to prove
\begin{equation*}
\dim_PE=\overline\dim_BE=t.
\end{equation*}

\bigskip

\section{The upper box dimensions and packing dimensions of the homogeneous Moran sets}
We begin to prove Theorem \ref{thm1},  the following lemmas are essential for the proof of Theorem \ref{thm1}.
\subsection{Some Lemmas}
We need Lemma \ref{lem01} to estimate the upper bounds of the upper box dimensions of the sets.
\begin{lem}\label{lem01}{\rm{(\cite{tc1981})}}
Let $E\in \mathcal{M}(I_{0},\{n_{k}\},\{c_{k}\})$, then $$\overline{\dim}_{B}E\le\inf \{s>0:\sum_{k=0}^{\infty} \sum_{\substack{\sigma \in D_k\\0\le i\le n_{k+1}}}\xi_{\sigma ,i}^{s}<\infty \} .$$
\end{lem}
\bigskip

To get the packing dimensions of the sets, we need to use  Lemma \ref{lem02}.
\begin{lem}\label{lem02}{\rm{(\cite{fk1990})}}
Let $E\subseteq \mathbb{R} $ be a bounded closed set, if $\overline{\dim}_{B}E= \overline{\dim}_{B}(V\cap E)$ for any open set $V\subseteq\mathbb{R}$ with $V \cap E\neq \emptyset$, then $\dim _{P}E=\overline{\dim}_{B}E$.
\end{lem}
\bigskip

Let $E\in \mathcal{M}(I_{0},\left\{n_{k}\right\},\left\{c_{k}\right\})$ satisfy the conditions of Theorem \ref{thm1}. To get (\ref{031}), the proof of Theorem \ref{thm1} is divided into  three parts.
\subsection{The  upper bound of the upper box dimension}

Let $t$ be the parameter in (\ref{t}), we give the following lemma and  corollary  to get $\overline{\dim }_BE\le t$.

\begin{lem}\label{lem05}
Let $E\in \mathcal{M}(I_{0},\left\{n_{k}\right\},\left\{c_{k}\right\})$, then
\begin{equation*}
\overline{\dim }_BE\le \limsup\limits_{k \to\infty}  \frac{\log n_1n_2\cdots n_kn_{k+1}}{-\log\delta_{k}+\log n_{k+1}}.
\end{equation*}
\end{lem}

\begin{proof}
If $ \limsup\limits_{k \to\infty}  \frac{\log n_1n_2\cdots n_kn_{k+1}}{-\log\delta_{k}+\log n_{k+1}}=1$, obviously, $\overline{\dim }_BE\le \limsup\limits_{k \to\infty}  \frac{\log n_1n_2\cdots n_kn_{k+1}}{-\log\delta_{k}+\log n_{k+1}}.$ Suppose that $ \limsup\limits_{k \to\infty}  \frac{\log n_1n_2\cdots n_kn_{k+1}}{-\log\delta_{k}+\log n_{k+1}}<1$, then we have that for any $1>l> \limsup\limits_{k \to\infty}  \frac{\log n_1n_2\cdots n_kn_{k+1}}{-\log\delta_{k}+\log n_{k+1}}$, there exist $\epsilon >0 $ and $k_0>0$ satisfying for any $k\ge k_0$, $$l> \frac{\log n_1n_2\cdots n_kn_{k+1}}{-\log\delta_{k}+\log n_{k+1}}+\epsilon,$$ which implies
\begin{equation}\label{051}
n_1n_2\cdots n_kn_{k+1}<(\frac{\delta _k}{n_{k+1}} )^{\epsilon -l} .
\end{equation}

For any $k\ge 0$ and $\sigma \in D_k$, by the Jensen inequality, we have
\begin{equation}\label{052}
\sum_{i=0}^{n_{k+1}} \xi _{\sigma ,i}^{l} \le (n_{k+1}+1)(\frac{1}{n_{k+1}+1} \sum_{i=0}^{n_{k+1}} \xi _{\sigma ,i})^{l}
= (n_{k+1}+1)[\frac{(1-n_{k+1}c_{k+1})\delta _k}{n_{k+1}+1} ]^{l} .
\end{equation}
Notice that $n_k\ge 2$  and $n_kc_k<1$ for any $k\ge 1$, then  $c_k<\frac{1}{2}$ and $\delta _k<2^{-k}$ for any $k\ge 1$. By (\ref{051}) and (\ref{052}), combing $n_k\ge 2$ for any $k\ge 1$, we have
\begin{equation}
\begin{aligned}
\sum_{k=k_0}^{\infty }\sum_{\substack {\sigma \in D_k\\0\le i\le n_{k+1}}}\xi _{\sigma ,i}^{l}
&\le \sum_{k=k_0}^{\infty }n_1n_2\cdots n_k(n_{k+1}+1) [\frac{(1-n_{k+1}c_{k+1})\delta _k}{n_{k+1}+1} ]^{l} \\
&\le 2\sum_{k=k_0}^{\infty }n_1n_2\cdots n_kn_{k+1} [\frac{(1-n_{k+1}c_{k+1})\delta _k}{n_{k+1}+1} ]^{l} \\
&\le 2\sum_{k=k_0}^{\infty }(\frac{\delta _k}{n_{k+1}} )^{\epsilon -l} (\frac{\delta _k}{n_{k+1}})^{l} \\
&\le 2\sum_{k=k_0}^{\infty }2^{-(k+1)\epsilon}\\
&<\infty. \nonumber
\end{aligned}
\end{equation}
Then by Lemma \ref{lem01}, we have $\overline{\dim }_BE\le l$. Notice that $l\in(\limsup\limits_{k \to\infty}  \frac{\log n_1n_2\cdots n_kn_{k+1}}{-\log\delta_{k}+\log n_{k+1}},1)$ is arbitrary, we obtain that $$\overline{\dim }_BE\le \limsup\limits_{k \to\infty}  \frac{\log n_1n_2\cdots n_kn_{k+1}}{-\log\delta_{k}+\log n_{k+1}}.$$

According to the first reconstruction of the homogeneous Moran sets, if $E\in \mathcal{M}(I_{0},\left\{n_{k}\right\},\left\{c_{k}\right\})$  satisfying the conditions of Theorem \ref{thm1}, then we have $E\in \mathcal{M}(I_{0}^*,\left\{n_{k}^*\right\},\left\{c_{k}^*\right\})$. Notice that $n_{k}^*=n_{k}$ for any $k\ge 1$, we have the following corollary by Lemma \ref{lem05}.

\begin{cor}\label{cor1}
Let $E\in \mathcal{M}(I_{0},\left\{n_{k}\right\},\left\{c_{k}\right\})$ satisfy the conditions of Theorem \ref{thm1}, then
\begin{equation*}
\overline{\dim }_BE\le \limsup\limits_{k \to\infty}  \frac{\log n_{1}^*n_{2}^*\cdots n_{k}^*n_{k+1}^*}{-\log\delta_{k}^*+\log n_{k+1}^*}=\limsup\limits_{k \to\infty}  \frac{\log n_1n_2\cdots n_kn_{k+1}}{-\log\delta_{k}^*+\log n_{k+1}}=t.
\end{equation*}
\end{cor}

\end{proof}

\subsection{The  lower bound of the upper box dimension}

For any $k\ge 1$, let $\{r_k\}$ be a sequence of positive numbers, which satisfies  $r_1=\frac{1-L_1-R_1}{n_1}$ and
\begin{equation}\label{053}
r_k=\frac{\delta_{k-1}-L_k-R_k}{n_k}=\frac{\delta_{k-1}^{*}}{n_k}
\end{equation}
for any $k\ge 2$. Obviously, $r_k\to 0$ as $k\to \infty$.

 Denote by $N_{\epsilon } (E)$ the minimal number of closed balls with diameter $\epsilon $ needed to cover $E$.  Next, we discuss separately according to the three conditions (A), (B) and (C) of Theorem \ref{thm1} to estimate $\frac{\log{N_{r_k } (E)} }{-\log {r_k}} $.

\textbf{Case 1:} Assume that condition (A) of Theorem \ref{thm1} is fulfilled.
Under this assumption, $\overline\nu_k\le w_1 \delta_k$ for any $k\ge 1$. Suppose that $B_k$ is the largest number of $k$-order basic intervals of $E$ which are intersected by a closed ball of diameter $r_k$ for any $k\ge 1$. Then
\begin{equation}
\begin{aligned}
(B_k-2)\delta _k\le r_k
&=\frac{\delta _{k-1}-L_k-R_k }{n_k} \\
&\le \frac{n_k\delta _k+(n_k-1)\overline\nu_k  }{n_k}\\
&\le \frac{n_k\delta _k+(n_k-1)w_1\delta _k  }{n_k}\\
&\le (w_1+1)\delta_k, \nonumber
\end{aligned}
\end{equation}
which implies $B_k\le w_1+3$. By (\ref{053}), we have
\begin{equation}\label{054}
\frac{\log{N_{r_k}(E)} }{-\log r_k} \ge \frac{\log {\frac{1}{2B_k}n_1n_2\cdots n_k }}{-\log r_k}
\ge \frac{\log {n_1n_2\cdots n_k}-\log {(2w_1+6)} }{-\log{\frac{\delta _{k-1}^{*} }{n_k} } } .
\end{equation}

\textbf{Case 2:} Assume that condition (B) of Theorem \ref{thm1} is fulfilled.
Under this assumption, $\overline\nu_k\le w_2 \underline{\nu }_k $ for any $k\ge 1$. This case can be divided into two subcases.

\textbf{(a)} $\overline\nu_k\le w_2 \underline{\nu }_k $ and $\underline{\nu }_k \le \delta_k $ for any $k\ge 1$. In this case, we obtain $\overline\nu_k\le w_2\delta_k$. Similarly to the proof of Case 1, we obtain
\begin{equation}\label{055}
\frac{\log{N_{r_k}(E)} }{-\log r_k}
\ge \frac{\log {n_1n_2\cdots n_k}-\log {(2w_2+6)} }{-\log{\frac{\delta _{k-1}^{*} }{n_k} } } .
\end{equation}

\textbf{(b)} $\overline\nu_k\le w_2 \underline{\nu }_k $ and $\underline{\nu }_k > \delta_k $ for any $k\ge 1$. In this case,
\begin{equation}
\begin{aligned}
\delta_{k-1}-L_k-R_k&\le n_k\delta_k+(n_k-1)\overline\nu_k\\
&\le n_k\underline{\nu }_k+ n_kw_2\underline{\nu }_k\\
&=(w_2+1)n_k\underline{\nu }_k.\nonumber
\end{aligned}
\end{equation}
By (\ref{053}), we obtain $$\frac{1}{2} \underline{\nu }_k\ge \frac{1}{2(w_2+1)} \cdot \frac{\delta_{k-1}-L_k-R_k}{n_k} =\frac{1}{2(w_2+1)}r_k.$$
Then the number of the $k$-order basic intervals of $E$ that a closed ball of diameter $\frac{1}{2(w_2+1)}r_k$ can intersect is at most 2, which yields that the number of the $k$-order basic intervals of $E$ that a closed ball of diameter $r_k$ can intersect is at most $2[2(w_2+1)+1]=2(2w_2+3):=Q_k$. Thus, we have
\begin{equation}\label{056}
\frac{\log{N_{r_k}(E)} }{-\log r_k}\ge \frac{\log \frac{1}{2Q_k}n_1n_2\cdots n_k  }{-\log r_k}
\ge \frac{\log {n_1n_2\cdots n_k}-\log {4(2w_2+3)} }{-\log{\frac{\delta _{k-1}^{*} }{n_k} } } .
\end{equation}

\textbf{Case 3:} Assume that condition (C) of Theorem \ref{thm1} is fulfilled.
Under this assumption, $n_k\underline{\nu }_k\ge w_3\delta_{k-1}$ for any $k\ge 1$. By (\ref{053}), we obtain
$$\frac{1}{2}\underline{\nu }_k\ge \frac{w_3}{2}\cdot \frac{\delta _{k-1}}{n_k}\ge \frac{w_3}{2}\cdot \frac{\delta _{k-1}-L_k-R_k}{n_k}=  \frac{w_3}{2}r_k.$$
Similarly to the proof of Case 2(b), we obtain
\begin{equation}\label{057}
\frac{\log{N_{r_k}(E)} }{-\log r_k}
\ge \frac{\log {n_1n_2\cdots n_k}-\log {4(\frac{2}{w_3} +1)} }{-\log{\frac{\delta _{k-1}^{*} }{n_k} } } .
\end{equation}

Since $r_k\to 0$ as $k\to \infty$, by (\ref{054}), (\ref{055}), (\ref{056}) and (\ref{057}), we can obtain that
\begin{align*}
\overline{\dim }_BE&=\limsup\limits _{\epsilon \to 0}\frac{\log{N_{\epsilon }(E)} }{-\log \epsilon }\\
&\ge \limsup\limits _{k \to \infty }\frac{\log{N_{r_k }(E)} }{-\log r_k }\\
&\ge \limsup\limits _{k \to \infty }\frac{\log{n_1n_2\cdots n_kn_{k+1}} }{-\log \frac{\delta _{k}^{*} }{n_{k+1}}  }\\
&=t.
\end{align*}

Combining Corollary \ref{cor1}, we have \begin{equation}\label{000}
\overline{\dim }_BE=t=\limsup\limits_{k \to\infty}  \frac{\log n_1n_2\cdots n_kn_{k+1}}{-\log(\delta_k-L_{k+1}-R_{k+1})+\log n_{k+1}}. \end{equation}

\subsection{The estimation of the packing dimension}

Let $V\subseteq\mathbb{R}$  be  an open set with $V \cap E\neq \emptyset$, then for any $x\in V\bigcap E$, there exist a real number $r>0$, a positive integer $k_0$ and $\sigma_0\in D_{k_0}$, such that
\begin{equation}\label{058}
x\in I_{\sigma _0}\subset B(x,r)\subset V .
\end{equation}

For any $\sigma_1,\sigma_2 \in D_{k_0}$,  It is easy to obtain that $I_{\sigma _1}\cap E\in \mathcal{M}(I_{\sigma _1},\left\{n_{k}\right\},\left\{c_{k}\right\})$ and $I_{\sigma _2}\cap E\in \mathcal{M}(I_{\sigma _2},\left\{n_{k}\right\},\left\{c_{k}\right\})$ with $k\ge k_0+1$.  It is also easy to obtain that $I_{\sigma _1}\cap E$ and $I_{\sigma _2}\cap E$ satisfy the conditions of Theorem \ref{thm1}. Then by (\ref{000}), we have
\begin{align*}
\overline{\dim }_B (I_{\sigma _1}\cap E)&=\overline{\dim }_B (I_{\sigma _2}\cap E)\\
&=\limsup\limits_{k\to \infty }\frac{\log{n_{k_0+1}n_{k_0+2}\cdots n_kn_{k+1}} }{-\log (c_{k_0+1}c_{k_0+2}\cdots c_k-L_{k+1}-R_{k+1})+\log n_{k+1}} ,
\end{align*}
 Notice that $\sigma_1$ and $\sigma_2$ are arbitrary, then we have
\begin{equation}\label{059}
\max\limits_{\sigma \in D_{k_0}}\overline{\dim }_B (I_{\sigma}\cap E)=\overline{\dim }_B (I_{\sigma _0}\cap E).
\end{equation}

By (\ref{058}), (\ref{059}) and the finite stability and monotonicity of the upper box dimension, we have
\begin{align*}
\overline{\dim }_BE&= \overline{\dim }_B[  \bigcup_{\sigma \in D_{k_0}} (I_{\sigma}\cap E)]\\
&=\max\limits_{\sigma \in D_{k_0}}\overline{\dim }_B(I_{\sigma}\cap E)\\
&=\overline{\dim }_B(I_{\sigma _0}\cap E)\\
&\le \overline{\dim }_B(V\cap E).
\end{align*}
By  the monotonicity of the upper box dimension, $\overline{\dim }_BE\ge \overline{\dim }_B(V\cap E)$, thus $\overline{\dim }_BE=\overline{\dim }_B(V\cap E)$. Therefore by Lemma \ref{lem02}, we have $\dim_PE=\overline{\dim }_BE=t=\limsup\limits_{k \to\infty}  \frac{\log n_1n_2\cdots n_kn_{k+1}}{-\log(\delta_k-L_{k+1}-R_{k+1})+\log n_{k+1}}$ and complete the proof of Theorem \ref{thm1}.

\bigskip

\section{The  quasisymmetric  packing-minimalities of the homogeneous Moran sets}
We begin to prove Theorem \ref{thm2},  we need the following lemmas.
\subsection{Some Lemmas}
We  use the mass distribution principle  to estimate the lower bounds of the packing dimensions of the quasisymmetric image sets.
\begin{lem}\label{lem03}\textmd{\rm (Mass distribution principle \rm{\cite{wen00}})}
Let $s\ge0$ and $\mu$ be a positive and finite Borel measure  on  $E\subseteq\mathbb{R}${\rm(}we call that $\mu$ is a mass distribution  on $E${\rm)}. If there exists $c>0$, such that for any $x\in E$, $$\liminf \limits _{r\to 0}\frac{\mu (B(x,r))}{r^s} \le c,$$ then $\dim_{P}E\ge s$.
\end{lem}
\bigskip

For any interval $I\subseteq\mathbb{R}$ and $\rho>0$, denote by $\rho I$   the interval concentric with $I$ of length $\rho \left|I\right|$.  Then we state the following lemma.
\begin{lem}\label{lem04}{\rm{(\cite{wjm93})}}
Let $f:\mathbb{R}\to \mathbb{R}$ be a \text{\rm1}-dimensional quasisymmetric mapping, then
for every pair of intervals $I'$ and $I$ satisfying $I'\subseteq I$,
 there exist $\lambda>0$, $K_{\rho}>0$ and $0<p\le 1\le q$ such that
    $$\lambda(\frac{|I^{'}|}{\left|I\right|})^{q}\le \frac{|f(I^{'})|}{\left|f(I)\right|}\le 4(\frac{|I^{'}|}{\left|I\right|})^{p},\quad
\frac{\left|f(\rho I)\right|}{\left|f(I)\right|}\le K_{\rho}.$$
\end{lem}
Lemma \ref{lem04} shows the relationship  between the lengths of the image sets and those of the original sets under quasisymmetric mappings, which is essential in the proof of the quasisymmetric minimality.

\subsection{The second reconstruction of the homogeneous Moran sets}\label{sc}
To prove Theorem \ref{thm2}, we need do the second equivalent reconstruction to the homogeneous Moran sets based on the first reconstruction in Section 4, as the following lemma.

\begin{lem}\label{lem06}
	 Let $E\in \mathcal{M}(I_{0},\left\{n_{k}\right\},\left\{c_{k}\right\})$ satisfy the conditions of Theorem {\rm\ref{thm2}}, $E(I^*_{0},\\\{n^*_{k}\}, \{c^*_{k}\})$ be the first reconstructed version of $E$. Then there exists a sequence $\{S_{m}\}_{m\ge 0}$,  which  consists of some closed sets  of decreasing length, such that $E=\cap_{k\ge0}E_{k}=\cap_{k\ge0}E_{k}^{*}=\cap_{m\ge0}S_{m}$. Moreover, $\{S_{m}\}_{m\ge 0}$ fulfills the four conditions below:
\begin{enumerate}
  \item[\textup{(1)}] For any $m\ge0$, $S_{m}$ is a union of finite closed intervals,   where  the finite closed intervals have disjoint interiors, which are called the basic intervals of $S_{m}$. Write $\mathcal{S}_{m}=\{B: B ~~~~\text{is a basic interval of}~~~~ S_{m}\}$;
  \item[\textup{(2)}] $\left\{E_{k}^{*}\right\}_{k\ge0}$ is a subsequence of $\left\{S_{m}\right\}_{m\ge 0}$, and there exists $\{m_k\}_{k\ge 0}$ which is an increasing sequence of positive integers such that $S_{m_{k}}=E_{k}^{*}$ for any $k\ge 0$;

  \item[\textup{(3)}] Let $w_1$, $w_2$ be the constants in Theorem {\rm\ref{thm2}}. If condition {\rm(A)} of Theorem {\rm\ref{thm2}} holds, then there is a constant $T \in \mathbb{Z}^{+}$ with $T>2(1+w_1)$ satisfying each basic interval of $S_{m-1}$ contains at most $T^{2}$ basic intervals of $S_{m}$ for any $m\ge1$; if condition {\rm(B)} of Theorem {\rm\ref{thm2}} holds, then there is a constant $T \in \mathbb{Z}^{+}$ with $T>2w_2$ satisfying each basic interval of $S_{m-1}$ contains at most $T^{2}$ basic intervals of $S_{m}$ for any $m\ge1$ ;
		\item[\textup{(4)}]If condition {\rm(A)} of Theorem {\rm\ref{thm2}} holds, then we have $\max_{I\in \mathcal{S}_m}\left | I \right |\le 2(w_1+1)\min_{I\in \mathcal{S}_m}\left | I \right |$ for any $m\ge0$; if condition {\rm(B)} of Theorem {\rm\ref{thm2}} holds, then
$\max_{I\in \mathcal{S}_m}\left | I \right |\le 2w_2\min_{I\in \mathcal{S}_m}\left | I \right |$ for any $m\ge0$.
\end{enumerate}
\end{lem}
\begin{proof}

We begin with the construction  of $\{S_{m}\}_{m\ge 0}$.

Let $T=\min\{x: x>2(1+w_1), x\in \mathbb{Z}^{+}\}$ if the conditon (A)  of Theorem {\rm\ref{thm2}} holds, and let $T=\min\{x: x>2w_2, x\in \mathbb{Z}^{+}\}$ if the conditon (B)  of Theorem {\rm\ref{thm2}} holds. For any $k\ge 1$, Let $i_{k}$ be a positive integer satisfying :
if $2\le n_{k}^{*}<T$, then $i_{k}=1$;
if $n_{k}^{*}\ge T$, then $i_{k}$ fulfills $T^{i_{k}}\le n_{k}^{*}<T^{i_{k}+1}$.
Define $m_{0}=0,\ m_{k}=\sum_{j=1}^{k}i_{j}$, then $m_k=m_{k-1}+i_k.$
	
	For any $k\ge 0$, let $S_{m_{k}}=E_{k}^{*}$, which implies $\mathcal{S}_{m_{k}}=\{I_{\sigma}^{*}:\sigma\in D_{k}\}$, that is,  all of the $k$-order first reconstructed basic intervals of $E$ constitute $S_{m_{k}}$.
Next, for any $k\ge 1$ and $m_{k-1}<m<m_{k}$, we construct $S_{m}$ .

\begin{enumerate}
		\item[\textup{(1)}] If $2\le n_{k}^{*}< T^{2}$, then $i_{k}=1$ and $m_k=m_{k-1}+1$, there exists no integer $m$ with $m_{k-1}< m<m_k$.
		\item[\textup{(2)}] If $n_{k}^{*}\ge T^{2}$, then $i_{k}\ge 2$, and there exist $a_{i_k}\in\left\{1,2,\cdots,T-1\right\}$ and $a_{j}\in\left\{0,1,\cdots,T-1\right\}$ for any $j\in\left\{0,1,\cdots,i_{k}-1\right\}$, satisfying
\begin{equation*}
		n_{k}^{*}=a_{0}+a_{1}T+a_{2}T^{2}+\cdots+a_{i_{k}-1}T^{i_{k}-1}+a_{i_k}T^{i_{k}}.
\end{equation*}
\end{enumerate}
	
For any $k\ge1$ and $\sigma\in D_{k-1}$, notice that  $S_{m_{k-1}}=E_{k-1}^{*}$, then $S_{m_{k-1}}$ has $N_{k-1}^{*}$ basic intervals,  and there are $n_{k}^{*}$ $k$-order first reconstructed basic intervals of $E$ contained in  every $I_{\sigma}^{*}\in \mathcal{S}_{m_{k-1}}$, denoted by $I_{\sigma*1}^{*},I_{\sigma*2}^{*},\cdots,I_{\sigma*n_{k}^{*}}^{*}$  from left to right.

Next, for any  $1\le i\le i_k-1$, we construct $S_{m_{k-1}+i}$.

 Given $l$ closed intervals $I_{1},I_{2},\cdots,I_{l}$, we write $[I_{1},I_{2},\cdots,I_{l}]$ for the minimal closed interval containing all of $I_{1},I_{2},\cdots,I_{l}$.

\begin{enumerate}
\item[\textup{(I)}] For any $I_{\sigma}^{*}\in \mathcal{S}_{m_{k-1}}$,
let $l_{1}=a_{1}+a_{2}T+\cdots+a_{i_{k}-1}T^{i_{k}-2}+a_{i_k}T^{i_{k}-1}$, then $n^*_k=Tl_1+a_0=a_0(l_1+1)+(T-a_0)l_1$.
We define $T$ closed subintervals of $I_{\sigma}^{*}$, which are denoted by $I_{1}^{\sigma,1}, I_{2}^{\sigma,1}, \cdots, I_{a_{0}}^{\sigma,1}, I_{a_{0}+1}^{\sigma,1}, \cdots, I_{T}^{\sigma,1}$ as follows. For $1\le i\le a_0$, write
$$I_{i}^{\sigma ,1}=\Big[I_{\sigma *\big(i(l_1+1)-l_1 \big)}^{*},I_{\sigma *\big(i(l_1+1)-l_1+1 \big)}^{*},\cdots ,I_{\sigma *\big(i(l_1+1) \big)}^{*}\Big];$$
for $a_0+1\le i\le T$, write
$$I_{i}^{\sigma ,1}=\Big[I_{\sigma *\big((i-1)l_1+a_0+1\big)}^{*},I_{\sigma *\big((i-1)l_1+a_0+2\big)}^{*},\cdots ,I_{\sigma *\big(il_1+a_0\big)}^{*}\Big].$$
Then for each $I_{i}^{\sigma,1}(1\le i\le a_0)$, there are $l_1+1$ $k$-order first reconstructed basic intervals of $E$ contained in $I_{i}^{\sigma,1}$, and for each $I_{i}^{\sigma,1}(a_0+1\le i\le T)$,  there are  $l_1$  $k$-order first reconstructed basic intervals of $E$ contained in $I_{i}^{\sigma,1}$.
Let $S_{m_{k-1}+1}=\bigcup_{\sigma \in D_{k-1}}(\bigcup_{i=1}^{T}I^{\sigma,1}_i)$, and every closed interval $I_{i}^{\sigma,1}(1\le i\le T)$ is referred to as a basic interval of $S_{m_{k-1}+1}$ in $I_{\sigma}^{*}$. Then each basic interval of $S_{m_{k-1}}$ contains $T$ basic intervals of $S_{m_{k-1}+1}$.

\item[\textup{(II)}] If $i_{k}=2$, then $m_{k}=m_{k-1}+2$, and $S_{m_{k-1}+1}$ is defined  as bove, $S_{m_{k-1}}=E_{k-1}^{*}$, $S_{m_{k}}=E_{k}^{*}$. Thus the construction of $S_{m_{k-1}+i}$  for any $1\le i\le i_k-1$ is finished.

\item[\textup{(III)}] If $i_{k}\ge3$, then we proceed to construct  $S_{m_{k-1}+2}$. Let $l_2=a_2+a_3T+\cdots +a_{i_k-1}T^{i_k-3}+a_{i_k}T^{i_k-2}$, then we have $l_{1}=Tl_{2}+a_{1}$, $n^*_k=T^2l_2+a_1T+a_0=a_0(Tl_2+a_1+1)+(T-a_0)(Tl_2+a_1)$.

For any $I^{\sigma,1}_{i}\in \mathcal{S}_{m_{k-1}+1}(\sigma\in D_{k-1}, 1\le i\le T)$, the construction is divided into the two cases as follows:

\textbf{(i):} If $1\le i\le a_{0}$,
which means each $I^{\sigma,1}_{i}$ contains $l_1+1$  $k$-order first reconstructed basic intervals of $E$,
and $l_1+1=Tl_2+a_1+1=(l_2+1)(a_1+1)+l_2(T-a_1-1)$. Thus, we can define $T$ closed subintervals of $I^{\sigma,1}_{i}$, which are denoted by $I_{i*1}^{\sigma,1}, I_{i*2}^{\sigma,1}, \cdots , I_{i*(a_1+1)}^{\sigma,1}, I_{i*(a_1+2)}^{\sigma,1}, \cdots, I_{i*T}^{\sigma,1}$ as follows. For $1\le j\le a_1+1$, write
$$I_{i*j}^{\sigma ,1}=\Big[I_{\sigma *\big((i-1)l_1+i+(j-1)(l_2+1) \big)}^{*},I_{\sigma *\big((i-1)l_1+i+(j-1)(l_2+1)+1 \big)}^{*},\cdots ,I_{\sigma *\big((i-1)l_1+i+j(l_2+1)-1 \big)}^{*} \Big];$$
for $a_1+2\le j\le T$, write
$$I_{i*j}^{\sigma ,1}=\Big[I_{\sigma *\big((i-1)l_1+i+(j-1)l_2+a_1+1 \big)}^{*},I_{\sigma *\big((i-1)l_1+i+(j-1)l_2+a_1+2 \big)}^{*},\cdots ,I_{\sigma *\big((i-1)l_1+i+jl_2+a_1 \big)}^{*} \Big].$$
Then for each $I_{i*j}^{\sigma,1}(1\le j\le a_1+1)$, there are $l_2+1$  $k$-order first reconstructed basic intervals of $E$ contained in $I_{i*j}^{\sigma,1}$,
and for each  $I_{i*j}^{\sigma,1}(a_1+2\le j\le T)$, there are $l_2$  $k$-order first reconstructed basic intervals of $E$  contained in $I_{i*j}^{\sigma,1}$.

\textbf{(ii):} If $a_{0}+1\le i\le T$, which means
each $I^{\sigma,1}_{i}$ contains $l_1$ $k$-order first reconstructed basic intervals of $E$,
and $l_1=Tl_2+a_1=a_1(l_2+1)+(T-a_1)l_2$. Thus, we can define $T$ closed subintervals of $I^{\sigma,1}_{i}$, which are denoted by $I_{i*1}^{\sigma,1}, I_{i*2}^{\sigma,1},\cdots, I_{i*a_1}^{\sigma,1}, I_{i*(a_1+1)}^{\sigma,1},\cdots, I_{i*T}^{\sigma,1}$ as follows. For $1\le j\le a_1$, write
$$I_{i*j}^{\sigma ,1}=\Big[I_{\sigma *\big((i-1)l_1+a_0+(j-1)l_2+j \big)}^{*},I_{\sigma *\big((i-1)l_1+a_0+(j-1)l_2+j+1 \big)}^{*},\cdots ,I_{\sigma *\big((i-1)l_1+a_0+j(l_2+1) \big)}^{*} \Big];$$
for $a_1+1\le j\le T$, write
$$I_{i*j}^{\sigma ,1}=\Big[I_{\sigma *\big((i-1)l_1+a_0+(j-1)l_2+a_1+1 \big)}^{*},I_{\sigma *\big((i-1)l_1+a_0+(j-1)l_2+a_1+2 \big)}^{*},\cdots ,I_{\sigma *\big((i-1)l_1+a_0+jl_2+a_1 \big)}^{*} \Big].$$
Then for each $I_{i*j}^{\sigma,1}(1\le j\le a_1)$, there are $l_2+1$  $k$-order first reconstructed basic intervals of $E$  contained in $I_{i*j}^{\sigma,1}$,
and for each $I_{i*j}^{\sigma,1}(a_1+1\le j\le T)$, there are $l_2$   $k$-order first reconstructed basic intervals of $E$  contained in $I_{i*j}^{\sigma,1}$.

For any $1\le d \le T$ and$(d-1)T+1\le h\le dT$, define $I^{\sigma,2}_{h}=I^{\sigma,1}_{d*\big(h-(d-1)T\big)}$. Let $$S_{m_{k-1}+2}=\bigcup_{\sigma \in D_{k-1}}\bigcup_{i=1}^{T} \bigcup_{j=1}^{T}I^{\sigma,1}_{i*j}=\bigcup_{\sigma \in D_{k-1}}\bigcup_{h=1}^{T^2}I^{\sigma,2}_{h},$$ and every closed interval  $I^{\sigma,1}_{i*j}(1\le j\le T)$ is referred to as a basic intervals of $S_{m_{k-1}+2}$ in $I_{i}^{\sigma,1}$. Then each basic interval of $S_{m_{k-1}+1}$ contains $T$ basic intervals of $S_{m_{k-1}+2}$.

\item[\textup{(IV)}] If $i_{k}=3$, then $m_{k}=m_{k-1}+3$, and $S_{m_{k-1}+1}$, $S_{m_{k-1}+2}$ are defined  as above, $S_{m_{k-1}}=E_{k-1}^{*}$, $S_{m_{k}}=E_{k}^{*}$. Thus the construction of $S_{m_{k-1}+i}$ for any $1\le i \le i_k-1$ is finished.

\item[\textup{(V)}]If $i_k\ge 4$, then $m_k=m_{k-1}+i_k$. If $S_{m_{k-1}+i-1}(3\le i\le i_k-1)$ has been constructed,  then we can repeat the procedure of the construction of $S_{m_{k-1}+i-1}$ from $S_{m_{k-1}+i-2}$ to  finish the construction of $S_{m_{k-1}+i}$ from $S_{m_{k-1}+i-1}$. Thus, each basic interval of $S_{m_{k-1}+j-1}\ (1\le j\le i_{k}-1)$ contains $T$ basic intervals of $S_{m_{k-1}+j}$, which yields that each basic interval of $S_{m_{k-1}}$ contains $T^{i_{k}-1}$ basic intervals of $S_{m_{k-1}+i_{k}-1}$.
 Note that  $m_k=m_{k-1}+i_k$ and $S_{m_{k}}=E_{k}^{*}$ for any $k\ge 1$, it follows that every basic interval of $S_{m_{k-1}}$ contains $n_k^{*}$ basic intervals of $S_{m_{k}}$, then every basic interval of $S_{m_{k-1}+i_{k}-1}$ contains at most $T^{2}$ basic intervals of $S_{m_{k}}$(Otherwise, if  there exists a basic interval of $S_{m_{k-1}+i_{k}-1}$ containing $T^{*}$ basic intervals of $S_{m_{k}}$ with $T^{*}>T^{2}$, then
every basic interval of $S_{m_{k-1}+i_{k}-1}$ contains $T^{*}$ or $T^{*}-1$ or $T^{*}+1$ basic intervals of $S_{m_{k}}$. Then we can obtain that $n_{k}^{*}\ge (T^{*}-1)\cdot T^{i_{k}-1}\ge T^2\cdot T^{i_{k}-1}=T^{i_k+1}$, which contradicts  $ n_k^{*}<T^{i_k+1}$. Similarly, we can obtain that  every basic interval of $S_{m_{k-1}+i_{k}-1}$ contains at least $T$ basic intervals of $S_{m_{k}}$ since $n_{k}^{*} \ge T^{i_{k}}$).

\end{enumerate}

    We complete the construction of $\{S_{m}\}_{m\ge 0}$ and show that  conditions (1)-(3) of Lemma \ref{lem06} are satisfied for $E\in \mathcal{M}(I_{0},\left\{n_{k}\right\},\left\{c_{k}\right\})$ satisfying the conditions of Theorem \ref{thm2}.

Finally, we consider  condition (4) of Lemma \ref{lem06}. Obviously, $\max_{I\in\mathcal{S}_{m_{k}}}\left|I\right|= \min_{I\in\mathcal{S}_{m_{k}}}\left|I\right|=\delta_{k}^*$ for any $k \ge 0$ since $S_{m_{k}}=E_{k}^{*}$ for any $k\ge 0$.

For any $k\ge 1$, $m_{k-1}< m <m_k$, let $\tilde{I}_{\min}, \tilde{I}_{\max}\in \mathcal{S}_m$ satisfy $\max_{I\in \mathcal{S}_m}|I|=|\tilde{I}_{\max}|$, $\min_{I\in \mathcal{S}_m}|I|=|\tilde{I}_{\min}|$.
Suppose that the number of the basic intervals of $S_{m_k}$ contained in  $\tilde{I}_{\max}$ is $\overline{M}$, and the number of the basic intervals of $S_{m_k}$ contained in $\tilde{I}_{\min}$ is $\underline{M}$. Then by the above construction, we have
\begin{equation}\label{061}
\overline{M}\le \underline{M}+1.
\end{equation}

If condition (A) of Theorem \ref{thm2} holds, then for any $k\ge 1$,
$
\overline\nu_k\le w_1\delta_k .
$
By (\ref{045}), we have
\begin{equation}\label{062}
\overline\nu^*_{k}=\overline\nu_k+L_{k+1}+R_{k+1}.
\end{equation}

Notice that $\underline{M}\ge T>2(w_{1}+1)>2$, combine (\ref{041}), (\ref{048}), (\ref{061}), (\ref{062}) , we have
\begin{align*}
\max_{I\in \mathcal{S}_m}|I|
&\le \overline{M}\delta_{k}^{*}+(\overline{M}-1)\overline{\nu}_{k}^{*}\\&= \overline{M}\delta_{k}^{*}+(\overline{M}-1)(\overline{\nu}_{k}+L_{k+1}+R_{k+1})\\
&\le \overline{M}\delta_{k}^{*}+(\overline{M}-1)(w_1 \delta_{k}+L_{k+1}+R_{k+1})\\
&\le (\underline{M}+1)\delta_{k}^{*}+\underline{M}[w_1 (\delta_{k}^{*}+L_{k+1}+R_{k+1})+L_{k+1}+R_{k+1}]\\
&=[(w_1 +1)\underline{M}+1]\delta_{k}^{*}+(w_1 +1)\underline{M}(L_{k+1}+R_{k+1}) \\
&\le (w_1 +1)[(\underline{M}+1)\delta_{k}^{*}+\underline{M}(L_{k+1}+R_{k+1})]\\
&\le 2(w_1 +1)[\underline{M}\delta_{k}^{*}+(\underline{M}-1)\underline{\nu }_{k}^{*}] \\
&\le 2(w_1 +1)\min_{I\in \mathcal{S}_m}|I|.
\end{align*}

If condition (B) of Theorem \ref{thm2} holds, then for any $k\ge 1$,
$\overline\nu_k\le w_2\underline\nu_k$, which implies $w_{2}\ge 1$.
By (\ref{045}) and (\ref{046}), we have
\begin{equation}\label{063}
\overline\nu^*_{k}=\overline\nu_k+L_{k+1}+R_{k+1}\le w_2\underline\nu_k+L_{k+1}+R_{k+1}\le w_2\underline\nu_{k}^{*}.
\end{equation}

Notice that $\underline{M}\ge T>2w_{2}\ge 2$, combine (\ref{061}), (\ref{063}) and $w_{2}\ge 1$, we have
\begin{align*}
\max_{I\in \mathcal{S}_m}|I|
&\le \overline{M}\delta_{k}^{*}+(\overline{M}-1)\overline{\nu}_{k}^{*}\\
&\le (\underline{M}+1)\delta_{k}^{*}+\underline{M}\overline{\nu}_{k}^{*}\\
&\le 2w_2[\underline{M}\delta_{k}^{*}+(\underline{M}-1)\underline{\nu }_{k}^{*}] \\
&\le 2w_2\min_{I\in \mathcal{S}_m}|I|.
\end{align*}

By  the above argument, condition (4) of Lemma \ref{lem06} is  satisfied for $E\in \mathcal{M}(I_{0},\left\{n_{k}\right\},\\ \left\{c_{k}\right\})$  satisfying the conditions of Theorem \ref{thm2}, which implies that we finish
 the proof of Lemma \ref{lem06}.

\end{proof}
\bigskip

\begin{rem}\label{rm6}
Without loss of generality,  we take $I_{0}^{*}=[0,1]$ in this paper, which implies $\delta^*_{0}=1$ and $S_{m_{0}}=E_{0}^{*}=[0,1]$.
\end{rem}

\bigskip

\subsection{Some marks and lemmas}

Let $E\in \mathcal{M}(I_{0},\left\{n_{k}\right\},\left\{c_{k}\right\})$  satisfy the conditions of Theorem {\rm\ref{thm2}}, and $\left\{S_{m}\right\}_{m\ge0}$ be the sequence in Lemma \text{\rm\ref{lem06}}.

For any $m\ge0$ and $\tilde{I}\in \mathcal{S}_m$, $\tilde{I}-(\tilde{I}\bigcap S_{m+1})$ is a union of some intervals with their interiors pairwise disjoint, where the leftmost and rightmost members  are half-open intervals or empty sets, others are open intervals, and these intervals are referred to as  the gaps of $\tilde{I}$. Denote by $\mathcal{G}_m$ the family of all the gaps of all basic intervals of $S_m$, which means  $\mathcal{G}_m=\{$The gaps of $\tilde{I}: \tilde{I}\in \mathcal{S}_m\}$.

For any  $\tilde{I}\in \mathcal{S}_m$, we write $\mathcal{G}(\tilde{I})=\left \{ \tilde{I}^{'}:\tilde{I}^{'}\subset \tilde{I},\tilde{I}^{'}\in \mathcal{G}_m \right \}$.
By the argument in subsection \ref{sc},  $\tilde{I}$ contains at most $T^2$ basic intervals of $S_{m+1}$, which implies that $\#(\mathcal{G}(\tilde{I}))\le T^2+1$ ($\#$ denotes the cardinality).

For any $\tilde{I}\in \mathcal{S}_m$, let $\tilde{I}_1,\tilde{I}_2,\cdots ,\tilde{I}_{N(\tilde{I})}$ be all basic intervals of $S_{m+1}$ contained in $\tilde{I}$, where $N(\tilde{I})$ is the number of the basic intervals of $S_{m+1}$ contained in $\tilde{I}$, then $N(\tilde{I})\le T^{2}$.

For any $m\ge 1, k\ge 1$ and $m_{k-1}<m \le m_k$, write

$$\overline{\Lambda}_m=\frac{\max\limits_{\tilde{I}\in\mathcal{S}_m}|\tilde{I}|}{\min\limits_{\tilde{I}\in\mathcal{S}_{m-1}}|\tilde{I}|},
~\quad~\underline{\Lambda}_m=\frac{\min\limits_{\tilde{I}\in\mathcal{S}_m}|\tilde{I}|}{\max\limits_{\tilde{I}\in\mathcal{S}_{m-1}}|\tilde{I}|};$$

$$\lambda _{m-1}=\max\left \{ \frac{|\tilde{I}^{'}|}{|\tilde{I}|}:\tilde{I}\in \mathcal{S}_{m-1}, \tilde{I}^{'}\in \mathcal{G}(\tilde{I}) \right \};$$

$$\tau _{m-1}=\min\left \{ \frac{\sum_{i=1}^{N(\tilde{I}) } |\tilde{I}_i| }{|\tilde{I}|}:\tilde{I}\in \mathcal{S}_{m-1} \right \};$$

$$\pi _m=\max\left \{ \frac{|\tilde{I}|}{|\tilde{J}|}:\tilde{I}\in \mathcal{S}_m, \tilde{J}\in \mathcal{S}_{m-1},\tilde{I}\subset \tilde{J} \right \}.$$

We need the following lemmas for further discussions.
\begin{lem}\label{lem07}
$\tau _m\ge 1-(T^{2}+1)\lambda _m$ for any $m\ge0$, .
\end{lem}
\begin{proof}
For any $\tilde{I}\in \mathcal{S}_m$ and any $\tilde{I}^{'}\in \mathcal{G}(\tilde{I})$, we have
$\lambda _m\ge \frac{|\tilde{I}^{'}|}{|\tilde{I}|}$. Then
\begin{equation*}
\sum_{\tilde{I}^{'}\in \mathcal {G}(\tilde{I})} \frac{|\tilde{I}^{'}|}{|\tilde{I}|} \le \sum_{\tilde{I}^{'}\in \mathcal {G}(\tilde{I})}\lambda _m=\#\big(\mathcal{G}(\tilde{I})\big)\lambda _m\le (T^{2}+1)\lambda _m,
\end{equation*}
and yields
\begin{equation*}
\frac{ \sum_{i=1}^{N(\tilde{I}) } |\tilde{I}_i|}{|\tilde{I}|}=\frac{|\tilde{I}|-\sum_{\tilde{I}^{'}\in \mathcal {G}(\tilde{I})}|\tilde{I}^{'}|}{|\tilde{I}|} \ge 1-(T^{2}+1)\lambda _m.
\end{equation*}

Notice that $\tilde{I}$ is arbitrary, then we obtain
\begin{equation*}
\tau _m=\min\left \{ \frac{ \sum_{i=1}^{N(\tilde{I}) } |\tilde{I}_i|}{|\tilde{I}|}:\tilde{I}\in \mathcal{S}_m \right \}\ge 1-(T^{2}+1)\lambda _m.
\end{equation*}

\end{proof}

\begin{lem}\label{lem08}
Let $\{a_m\}_{m\ge 0}$ be a sequence of non-negative real numbers, if there exists a sequence $\{t_m\}_{m\ge 0}$ of non-negative integers such that
\begin{equation*}
\lim_{m \to \infty} \frac{1}{t_m} \sum_{i=0}^{t_m-1} a_i=0.
\end{equation*}
Then for any $\varepsilon \in (0,1)$, we have
\begin{equation*}
\lim\limits_{m\to\infty} \frac{\tilde{S}_\varepsilon (t_m)}{t_m}=1,\end{equation*}
where  $\tilde{S}_\varepsilon (m)=\#\big(\left \{ 0\le i\le m-1: a_i<\varepsilon \right \} \big)$  $(\#$ denotes the cardinality $)$ for any $m\ge 1$.
\end{lem}

\begin{proof}
For any $\varepsilon \in (0,1)$, we have
\begin{equation*}\sum_{i=0}^{t_m-1} a_i=\sum_{a_i<\varepsilon, 0\le i\le t_m-1}a_i+\sum_{a_i\ge \varepsilon, 0\le i\le t_m-1}a_i\ge \sum_{a_i\ge \varepsilon, 0\le i\le t_m-1}a_i \ge \big(t_m-\tilde{S}_\varepsilon (t_m) \big)\varepsilon.\end{equation*}

Since
\begin{equation*}
\lim\limits_{m\to\infty} \frac{1}{t_m} \sum_{i=0}^{t_m-1} a_i=0,
\end{equation*}
we obtain
\begin{align*}
1\ge \limsup\limits_{m\to\infty}\frac{\tilde{S}_\varepsilon (t_m)}{t_m}&\ge \liminf\limits_{m\to\infty} \frac{\tilde{S}_\varepsilon (t_m)}{t_m}\\&=1-\limsup\limits_{m\to\infty} \frac{t_m-\tilde{S}_\varepsilon (t_m)}{t_m}\\&\ge 1-\limsup\limits_{m\to\infty}\frac{1}{t_m\varepsilon} \sum_{i=0}^{t_m-1} a_i\\&=1,
\end{align*}
which yields  $\lim\limits_{m\to\infty} \frac{\tilde{S}_\varepsilon (t_m)}{t_m}=1.$
\end{proof}

We can obtain the  follwing lemma by equation (\ref{031}) in Theorem \ref{thm1} of this paper.
\begin{lem}\label{lem09}

Let $E\in \mathcal{M}(I_{0},\left\{n_{k}\right\},\left\{c_{k}\right\})$ satisfy the conditions of Theorem {\rm\ref{thm2}}, $\{S_{m}\}_{m\ge 0}$ and $\{m_{k}\}_{k\ge 0}$ be the sequences obtained from Lemma {\rm\ref{lem06}}.

If $\dim_{P}E=1$, then there exists a subsequence $\{b_{k}\}_{k\ge 0}$ of $\{m_{k}-1\}_{k\ge 0}$, such that
\begin{enumerate}
\item[\textup{($\bar a$)}]$\lim\limits_{k\to \infty}\frac{1}{b_k} \log_{T}{l(S_{b_k})}=0$, where $l(S_{m})$ is the sum of the lengths of all  basic intervals of $S_{m}$ for any $m\ge 0$;
\item[\textup{($\bar b$)}]$\lim\limits_{k\to \infty}\frac{1}{b_k}\sum_{t=0}^{b_k-1}\lambda_t=0$;
\item[\textup{($\bar c$)}]$\lim\limits_{k\to \infty}\frac{1}{b_k}\sum_{j=0}^{b_k-1}\log_{T}{\tau _j}=0$;
\item[\textup{($\bar d$)}]there is a real number $\beta\in (0,1)$, such that $\lim\limits_{k\to \infty}\frac{S_\beta  (b_k)}{b_k} =1$, where for any $\varepsilon\in (0,1)$ $S_\varepsilon (b_k)=\#(\left \{ 1\le j\le b_k:\pi _j<\varepsilon  \right \})$ $(\#$ denotes the cardinality $)$.
\end{enumerate}

\end{lem}

\begin{proof}
($\bar a$)
For any $k\ge 1$, we consider $l(S_{m_k-1})$ and $l(S_{m_{k-1}})$. Since $S_{m_{k-1}}=E_{k-1}^*$, we have $l(S_{m_{k-1}})=l(E_{k-1}^*)=N_{k-1}^* \delta_{k-1}^*$ by (\ref{049}). According to the construction of $\{S_m\}_{m\ge 0}$, to obtain $S_{m_k-1}$ from $S_{m_{k-1}}$, we need to  remove a half-open
interval of length $L_{k+1}$ and a half-open interval of length $R_{k+1}$ from each basic interval of $S_{m_{k-1}}$, and also need to  remove $[\sum_{j=0}^{i_k-2}T^{j}(T-1)  ]N_{k-1}^{*} =(T^{i_k-1}-1 )N_{k-1}^{*}$ open intervals whose lengths are at most $\overline{\nu }_{k}^{*} $ from $S_{m_{k-1}}$.

If condition {\rm(A)} of Theorem {\rm\ref{thm2}} holds, then $\overline{\nu }_{k}\le w_1\delta _{k} $. Combing $T^{i_k}\le n_{k}^{*}<T^{i_k+1} , w_1>0$,  (\ref{041}), (\ref{045}) and  (\ref{048}), we have
\begin{equation}\label{064}
\begin{aligned}
l(S_{m_k-1})
&\ge l(S_{m_{k-1}})-N_{k-1}^{*}[(L_{k+1}+R_{k+1})+(T^{i_k-1}-1)\overline\nu_{k}^{*}]\\
&\ge l(S_{m_{k-1}})-T^{i_k-1}N_{k-1}^{*}\overline\nu_{k}^{*}\\
&= l(S_{m_{k-1}})-T^{i_k-1}N_{k-1}^{*}(\overline\nu_{k}+L_{k+1}+R_{k+1})\\
&\ge l(S_{m_{k-1}})-T^{i_k-1}N_{k-1}^{*}(w_1\delta_{k}+L_{k+1}+R_{k+1})\\
&\ge l(S_{m_{k-1}})-T^{i_k-1}N_{k-1}^{*}(1+w_1)(\delta_{k}+L_{k+1}+R_{k+1})\\
&\ge l(S_{m_{k-1}})-\frac{n_{k}^*}{T}N_{k-1}^{*}(1+w_1)(\delta_{k}^{*}+2(L_{k+1}+R_{k+1}))\\
&\ge l(S_{m_{k-1}})-\frac{2(1+w_1)}{T}N_{k-1}^{*}\delta_{k-1}^{*}\\
&=\big(1- \frac{2(1+w_1)}{T}\big) l(S_{m_{k-1}})  .
\end{aligned}
\end{equation}

Since $\dim_PE=1$, then by equation (\ref{031}) in  Theorem \ref{thm1} and (\ref{041}), we have
\begin{equation*}\limsup\limits_{k \to\infty}  \frac{\log_{T} N_{k+1}^{*}}{-\log_{T}(\frac{\delta_{k}^*}{n_{k+1}^*} )}=\limsup\limits_{k \to\infty}  \frac{\log_{T} n_1n_2\cdots n_kn_{k+1}}{-\log_{T}(\delta_k-L_{k+1}-R_{k+1})+\log_{T} n_{k+1}}=1.\end{equation*}
Thus, \begin{equation*}\limsup\limits_{k \to\infty} \frac{\log_{T}{N_{k}^{*} } }{\log_{T}{N_{k}^{*}}-\log_{T}{N_{k-1}^{*}\delta_{k-1}^{*}}}=\limsup\limits_{k \to\infty} \frac{\log_{T}{N_{k}^{*} } }{-\log_{T}(\frac{\delta_{k-1}^*}{n_{k}^*} )}=1,\end{equation*}
which implies
\begin{equation}\label{065}\limsup\limits_{k\to\infty}(N_{k-1}^{*}\delta_{k-1}^{*})^{\frac{1}{\log_{T}N_{k}^{*}}}=1.\end{equation}

Since $T^{i_k}\le n_{k}^{*} <T^{i_k+1}$, $m_{k}=\sum_{j=1}^{k}i_j $, we have $N_{k}^{*}=\prod_{i=1}^{k}n_{i}^{*}<T^{m_k+k}\le T^{2m_k}     $, which yields $\log_{T}{N_{k}^{*} }\le  2m_k$.
Notice that $m_k-1\ge \frac{m_k}{2} \ge \frac{\log_{T}{N_{k}^{*} } }{4}$ by $m_k\ge 2$ if $k$ is large enough, then by (\ref{064}) and (\ref{065}), we obtain
\begin{equation}\label{066}
\begin{aligned}
\limsup \limits_{k\to \infty }[l(S_{m_k-1})]^{\frac{1}{m_k-1} }
&\ge \limsup \limits_{k\to \infty }[(1-\frac{2(1+w_1)}{T} )l(S_{m_{k-1}})]^{\frac{1}{m_k-1} }\\
&\ge  \limsup \limits_{k\to \infty }[(1-\frac{2(1+w_1)}{T} )(N_{k-1}^{*}\delta_{k-1}^{*} )]^{\frac{4}{\log_{T}{N_{k}^{*}} } }\\
&=1.
\end{aligned}
\end{equation}

On the other hand, for any $m\ge 0$, we have $l(S_m)\le 1$ by Remark \ref{rm6}, then
\begin{equation}\label{067}\limsup \limits_{k\to \infty }[l(S_{m_k-1})]^{\frac{1}{m_k-1} }\le \limsup \limits_{k\to \infty }[l(S_{m_k-1})]^{0}=1.\end{equation}
Combining (\ref{066}) and (\ref{067}), we obtain \begin{equation*}\limsup \limits_{k\to \infty }[l(S_{m_k-1})]^{\frac{1}{m_k-1} }=1.\end{equation*}
Therefore, there exists a subsequence $\{b_k\}_{k\ge 0}$ of $\{m_k-1\}_{k\ge 0}$, such that $$\lim_{k\to \infty }[l(S_{b_k})]^{\frac{1}{b_k} }=1,$$ which yields $\lim_{k\to \infty }\frac{1}{b_k}\log_{T}l(S_{b_k})=0$.

If condition {\rm(B)} of Theorem {\rm\ref{thm2}} holds, then $\overline{\nu }_k\le w_2\underline{\nu }_k  $. Combing $n_{k}^*\ge 2, w_2\ge 1$, $T^{i_k}\le n_{k}^*<T^{i_k+1}$, (\ref{048}) and (\ref{063}), we have
\begin{equation}
\begin{aligned}
l(S_{m_k-1})
&\ge l(S_{m_{k-1}})-N_{k-1}^{*}[(L_{k+1}+R_{k+1})+(T^{i_k-1}-1)\overline\nu_{k}^{*}]\\
&\ge l(S_{m_{k-1}})-T^{i_k-1}N_{k-1}^{*}\overline\nu_{k}^{*}\\
&\ge l(S_{m_{k-1}})-\frac{n_{k}^*}{T}N_{k-1}^{*}\overline\nu_{k}^{*}\\
&\ge l(S_{m_{k-1}})-\frac{2(n_{k}^*-1)}{T}N_{k-1}^{*} \overline\nu_{k}^{*}\\
&\ge l(S_{m_{k-1}})-\frac{2w_2}{T}(n_{k}^*-1)N_{k-1}^{*} \underline\nu_{k}^{*}\\
&\ge l(S_{m_{k-1}})-\frac{2w_2}{T}N_{k-1}^{*}\delta_{k-1}^{*}\\
&\ge\big(1- \frac{2w_2}{T}\big) l(S_{m_{k-1}})  .\nonumber
\end{aligned}
\end{equation}
Then by a similar argument to the case that the condition {\rm(A)} of Theorem {\rm\ref{thm2}} holds(replace $(1+w_1)$ by $w_2$), we obtain that there exists a subsequence $\{b_k\}_{k\ge 0}$ of $\{m_k-1\}_{k\ge 0}$, such that $$\lim_{k \to \infty }[l(S_{b_k})]^{\frac{1}{b_k} }=1,$$ which yields $\lim_{k \to \infty }\frac{1}{b_k}\log_{T}{l(S_{b_k})}=0 $.

($\bar b$) If condition {\rm(A)} of Theorem {\rm\ref{thm2}} holds, let $\{b_k\}_{k\ge 0}$ be the sequence in ($\bar a$), then for any $k\ge 0$, there exists $p_k\ge 1$ such that $b_k=m_{p_k}-1$. Notice that for any $1\le t\le p_k-1$ and $\sigma \in D_{t-1}$, $I_{\sigma}^*$ contains $n_{t}^*$ $t$-order  first reconstructed basic intervals of $E$ of length $\delta_{t}^*$ and $n_{t}^*+1$ gaps, and we have $\delta _0^{*}=1$ by Remark \ref{rm6}, then we obtain
\begin{align*}
&\frac{1}{b_k} \log_{T}{\prod_{t=1}^{p_{k}-1}\frac{\delta _{t-1}^{*}-\Big[\sum_{l=1}^{n_{t}^*-1}\xi _{\sigma ,l}^{*}+L_{t+1}+R_{t+1}   \Big]}{\delta _{t-1}^{*}}}\\
&=\frac{1}{b_k} \log_{T}{\prod_{t=1}^{p_{k}-1}\frac{n_{t}^*\delta _t^{*}}{\delta _{t-1}^{*}} }\\
&=\frac{1}{b_k} \log_{T}{(n_{1}^*n_{2}^*\cdots n_{p_{k}-1}^*\delta _{p_{k}-1}^{*})}-\frac{1}{b_k}\log_{T}{\delta _0^{*}}\\
&=\frac{1}{b_k} \log_{T}{\big(l(S_{m_{p_{k}-1}})\big)}.
\end{align*}
Thus, by ($\bar a$), we obtain \begin{align*}
\lim_{k\to\infty}\frac{1}{b_k} \log_{T}{\prod_{t=1}^{p_{k}-1}\frac{\delta _{t-1}^{*}-\Big[\sum_{l=1}^{n_{t}^*-1}\xi _{\sigma ,l}^{*}+L_{t+1}+R_{t+1} \Big]}{\delta _{t-1}^*}}&=
\lim_{k\to\infty}\frac{1}{b_k} \log_{T}{l(S_{m_{p_{k}-1}})}\\&=\lim_{k\to\infty}\frac{1}{b_k} \log_{T}{l(S_{b_k})}\\&=0.\end{align*}

Notice that for any $x\in [0,1)$, we have $\log_{T}{(1-x)}\le -x $, then
 \begin{align*}
&\frac{1}{b_k} \log_{T}
{\prod_{t=1}^{p_{k}-1}\frac{\delta _{t-1}^{*}-\Big[\sum_{l=1}^{n_{t}^*-1}\xi _{\sigma ,l}^{*}+L_{t+1}+R_{t+1} \Big]}{\delta _{t-1}^{*}}}\\
&\le \frac{1}{b_k}\sum_{t=1}^{p_{k}-1} -
\Big[\frac{\sum_{l=1}^{n_{t}^*-1}\xi _{\sigma ,l}^{*}+L_{t+1}+R_{t+1} }{\delta _{t-1}^{*}} \Big ]\\
&\le 0,
\end{align*}
which yields $$\lim_{k\to\infty}\frac{1}{b_k}\sum_{t=1}^{p_{k}-1} \frac{\sum_{l=1}^{n_{t}^*-1}\xi _{\sigma ,l}^{*}+L_{t+1}+R_{t+1}}{\delta _{t-1}^*}=0.$$
Since $\overline\nu^*_{t}=\max\limits_{\sigma\in D_{t-1}, 1\leq i\leq n_{t}^*-1}\xi_{\sigma,i}+L_{t+1}+R_{t+1}\le \sum_{l=1}^{n_{t}^*-1}\xi _{\sigma ,l}^{*}+L_{t+1}+R_{t+1}$ for any $1\le t\le p_k-1$, we have
\begin{equation}\label{068}
\lim_{k\to\infty}\frac{1}{b_k}\sum_{t=1}^{p_{k}-1}\frac{\overline{\nu } _{t}^{*}}{\delta _{t-1}^*}=0.
\end{equation}

We begin to estimate  $\lambda _m$ for any $m\ge 0$. Suppose $k$ is a positive integer with $m_{k-1}\le m< m_k$. If $\tilde{I}\in \mathcal {S}_{m_k-1}$, according to the reconstruction process, $\tilde{I}$ contains at least 2 basic intervals of $S_{m_k}$ and 1 gap whose length is at least $L_{k+1}+R_{k+1}$, then $|\tilde{I}|\ge 2\delta_{k}^{*} +L_{k+1}+R_{k+1}$.
If $\tilde{I}\in \mathcal {S}_{m_k-2}$, $\tilde{I}$ contains at least $2T$ basic intervals of $S_{m_k}$ and $2T-1$ gaps whose length is at least $L_{k+1}+R_{k+1}$, then $|\tilde{I}|\ge 2T\delta_{k}^{*} +(2T-1)(L_{k+1}+R_{k+1})$.
If $\tilde{I}\in \mathcal {S}_{m_k-3}$, $\tilde{I}$ contains at least $2T^2$ basic intervals of $S_{m_k}$ and $2T^2-1$ gaps whose length is at least $L_{k+1}+R_{k+1}$, then $|\tilde{I}|\ge 2T^2\delta_{k}^{*} +(2T^2-1)(L_{k+1}+R_{k+1})$.
From this, we can infer that if $r\in  \{ 1,2,\cdots ,m_k-m_{k-1} \}$, then we have for any $\tilde{I}\in \mathcal {S}_{m_k-r}$, $\tilde{I}$ contains at least $2T^{r-1}$ basic intervals of $S_{m_k}$ and $2T^{r-1}-1$ gaps whose length is at least $L_{k+1}+R_{k+1}$, which implies $|\tilde{I}|\ge 2T^{r-1}\delta_{k}^{*} +(2T^{r-1}-1)(L_{k+1}+R_{k+1})$. Since $T\ge 2$, we have $|\tilde{I}|\ge 2^{r}\delta_{k}^{*} +(2^{r}-1)(L_{k+1}+R_{k+1})\ge 2^{r-1}(\delta_{k}^{*} +L_{k+1}+R_{k+1})$.

Besides,  we have $|\tilde{I}^{'}|\le \overline{\nu } _{k}^{*}$ for any $\tilde{I}^{'} \in \mathcal {G}_{m}$. Then for any $r\in  \{ 1,2,\cdots ,m_k-m_{k-1} \}$, we have \begin{equation}\label{069}
\lambda _{m_k-r}\le \frac{\overline{\nu } _{k}^{*}}{2^{r-1}\big(\delta_{k}^{*}+L_{k+1}+R_{k+1} \big)}.\end{equation}
Therefore, \begin{equation}\label{0610} \sum_{m=m_{k-1}}^{m_k-1}\lambda _m=\sum_{r=1}^{i_k} \lambda _{m_k-r}
\le \frac{\overline{\nu } _{k}^{*}}{\delta_{k}^{*}+L_{k+1}+R_{k+1}} \sum_{r=0}^{i_k-1} \frac{1}{2^{r} }
\le \frac{2\overline{\nu } _{k}^{*}}{\delta_{k}^{*}+L_{k+1}+R_{k+1}}.\end{equation}
Thus, for any $k\ge 1$, we have \begin{equation}\label{0611}  \sum_{t=0}^{m_k-1} \lambda _t\le 2\sum_{t=1}^{k}\frac{\overline{\nu} _{t}^{*}}{\delta _{t}^{*}+L_{t+1}+R_{t+1}}. \end{equation}

For any $\epsilon >0$, there exists $\bar{\theta} >0$ satisfying \begin{equation}\label{0612} 0<\frac{1+w_1}{\log_{T}{\frac{1}{(1+w_1)\bar{\theta}} }-1 } <\frac{\epsilon }{4}. \end{equation}
For $t\ge 1$ with $\frac{\delta _{t}^{*}+L_{t+1}+R_{t+1}}{\delta _{t-1}^{*}} <\bar{\theta}$, since the condition {\rm(A)} of Theorem {\rm\ref{thm2}} holds, which means $\overline{\nu }_{t}\le w_1\delta _t = w_1(\delta _{t}^{*}+L_{t+1}+R_{t+1})$, by (\ref{044}), we have
\begin{align*}
\delta _{t-1}^{*}&= \sum_{l=1}^{n_{t}^{*}-1}\xi_{\sigma,l}^*+L_{t+1}+R_{t+1}+n_{t}^{*}\delta _{t}^{*}\\
&=\sum_{l=1}^{n_{t}^{*}-1}\xi_{\sigma,l}+n_{t}^{*}(L_{t+1}+R_{t+1})+n_{t}^{*}\delta _{t}^{*}\\
&\le (n_{t}^{*}-1)\overline{\nu }_{t}+n_{t}^{*}(L_{t+1}+R_{t+1}+\delta _{t}^{*})  \\
&\le (n_{t}^{*}-1)w_1(L_{t+1}+R_{t+1}+\delta _{t}^{*})+n_{t}^{*}(L_{t+1}+R_{t+1}+\delta _{t}^{*})  \\
&\le (1+w_1)n_{t}^{*}(L_{t+1}+R_{t+1}+\delta _{t}^{*}),
\end{align*}
then $(1+w_1)n_{t}^{*}\bar{\theta}\ge (1+w_1)n_{t}^{*} \frac{\delta _{t}^{*}+L_{t+1}+R_{t+1}}{\delta _{t-1}^{*}}\ge 1.$
Since $i_t\ge \log_{T}{n_{t}^{*}-1}$, by (\ref{0612}), we obtain \begin{equation}\label{0613} \frac{1+w_1}{i_t} <\frac{\epsilon }{4}.\end{equation}

By (\ref{068}), there is a positive integer $N>0$, such that for any $k\ge N$, \begin{equation}\label{0614} \frac{1}{b_k} \sum_{t=1}^{p_{k}-1} \frac{\overline{\nu } _{t}^{*}}{\delta _{t-1}^{*}} <\frac{\epsilon \bar{\theta} }{4}.\end{equation}
For $t\ge 1$, notice that $\overline{\nu }_{t}\le w_1\delta _{t}$, we have
\begin{equation}\label{0615} \overline{\nu }_{t}^{*}=\overline{\nu }_{t}+L_{t+1}+R_{t+1}\le (1+w_1)\delta _{t}=(1+w_1)(\delta _{t}^{*}+L_{t+1}+R_{t+1}).\end{equation}
Notice that $m_k=\sum_{j=1}^{k} i_j$, $b_k=m_{p_k}-1\ge m_{p_k-1}$, by (\ref{0613}), (\ref{0614}) and (\ref{0615}),  we obtain that if  $k\ge N$, then\begin{align*}
&\frac{1}{b_k} \sum_{t=1}^{p_{k}-1} \frac{\overline{\nu } _{t}^{*}}{\delta _{t}^{*}+L_{t+1}+R_{t+1}}\\
&\le \frac{1}{b_k} \Big[\sum_{\substack{t=1 \\
 {\frac{\delta _{t}^{*}+L_{t+1}+R_{t+1}}{\delta _{t-1}^{*}}<\bar{\theta}  }}}^{ p_{k}-1}\big(1+w_1 \big)+\sum_{\substack{t=1 \\
{\frac{\delta _{t}^{*}+L_{t+1}+R_{t+1}}{\delta _{t-1}^{*}}\ge \bar{\theta} }}}^{p_{k}-1}\frac{\overline{\nu } _{t}^{*}}{\delta _{t}^{*}+L_{t+1}+R_{t+1}} \Big]\\
&\le \frac{1}{b_k}\sum_{t=1}^{p_{k}-1}\frac{\epsilon i_t}{4} + \frac{1}{b_k}\sum_{t=1}^{p_{k}-1}\frac{\overline{\nu } _{t}^{*}}{\delta _{t-1}^{*}}\cdot \frac{1}{\bar{\theta}}\\
&<\frac{\epsilon }{4} +\frac{\epsilon }{4}\\
&=\frac{\epsilon }{2},\end{align*}
which implies $\frac{1}{b_k} \sum_{t=1}^{p_{k}-1} \frac{\overline{\nu } _{t}^{*}}{\delta _{t}^{*}+L_{t+1}+R_{t+1}} \to 0$ as $k\to \infty$. Then by (\ref{0611}), we obtain \begin{equation}\label{0616} \frac{1}{b_k} \sum_{t=0}^{m_{p_{k}-1}-1}\lambda _t\to 0\end{equation} as $k\to \infty$.

Notice that if $m_{p_{k}}=m_{p_{k}-1}+1$, then $b_k=m_{p_{k}-1}$, which yields $b_k-1=m_{p_{k}-1}-1$, therefore \begin{equation}\label{0617} \frac{1}{b_k} \sum_{t=0}^{b_k-1}\lambda _t= \frac{1}{b_k}\sum_{t=0}^{m_{p_{k}-1}-1}\lambda _t. \end{equation} If $m_{p_{k}}>m_{p_{k}-1}+1$, then $b_k-1\ge m_{p_{k}-1}$, which implies
\begin{equation}\label{0618}
\frac{1}{b_k} \sum_{t=0}^{b_k-1}\lambda _t=\frac{1}{b_k}\big(\sum_{t=0}^{m_{p_{k}-1}-1}\lambda _t+\sum_{t=m_{p_{k}-1}}^{b_k-1}\lambda _t \big)= \frac{1}{b_k}\sum_{t=0}^{m_{p_{k}-1}-1}\lambda _t+\frac{1}{b_k}\sum_{t=m_{p_{k}-1}}^{b_k-1}\lambda _t. \end{equation}
By (\ref{0610}) and (\ref{0615}), we have \begin{equation}\label{0619} \sum_{t=m_{p_{k}-1}}^{b_k-1}\lambda _t\le
\sum_{t=m_{p_{k}-1}}^{m_{p_{k}}-1}\lambda _t\le\frac{2\overline{\nu } _{p_{k}}^{*}}{\delta _{p_{k}}^{*}+L_{p_k+1}+R_{p_k+1}}
\le 2\big(1+w_1 \big).\end{equation}
Combining (\ref{0616}), (\ref{0618}) and (\ref{0619}), we have $$\lim\limits_{k\to \infty}\frac{1}{b_k}\sum_{t=0}^{b_k-1}\lambda_t=0.
$$

If condition {\rm(B)} of Theorem {\rm\ref{thm2}} holds, then $\overline{\nu }_k\le w_2\underline{\nu }_k  $. For any $k\ge1 $ and $m_{k-1}\le m< m_k$, write $\varphi _{m}=\min \{\frac{\xi_{\sigma,l}^*}{|I|}:I\in \mathcal{S}_m,\sigma \in D_{k-1}, 1\le l\le n_k-1  \} $. By (\ref{048}) and (4) of Lemma {\ref{lem06}}, we have \begin{equation}\label{0620}\lambda _{m}\le 2 w_{2}^{2}\cdot \varphi _{m}. \end{equation}
Let $\{b_k\}_{k\ge 0}$ be the sequence in ($\bar a$), then for any $0\le t\le b_k-1$ and $I\in \mathcal{S}_t$, we have $\frac{l(S_{t+1} )}{l(S_t)} \le \frac{|I|-\varphi_t\cdot |I|}{|I|} =1-\varphi_t$, which yields \begin{equation}\label{0621} l(S_{b_k} )\le \prod_{t=0}^{b_k-1}(1-\varphi_t) . \end{equation}

Notice that $\log_{T}{(1-x)} \le -x$ for any  $x\in [0,1)$, combining ($\bar a$) and (\ref{0621}), we obtain
$$0\ge \lim_{k \to \infty} -\frac{1}{b_k} \sum_{t=0}^{b_k-1}\varphi _{t}\ge
\lim_{k \to \infty} \frac{1}{b_k} \sum_{t=0}^{b_k-1}\log_{T}{(1-\varphi _{t})}\ge
\lim_{k \to \infty} \frac{\log_{T}{l(S_{b_k})}}{b_k}=0,$$
which implies \begin{equation}\label{0622}  \lim_{k \to \infty} \frac{1}{b_k} \sum_{t=0}^{b_k-1}\varphi _{t}=0. \end{equation}

By (\ref{0620}) and (\ref{0622}), we have $$0\le \lim_{k \to \infty} \frac{1}{b_k} \sum_{t=0}^{b_k-1}\lambda _{t}\le 2w_{2}^{2}\cdot \lim_{k \to \infty} \frac{1}{b_k} \sum_{t=0}^{b_k-1}\varphi _{t} =0,$$
which yields $$\lim_{k \to \infty} \frac{1}{b_k} \sum_{t=0}^{b_k-1}\lambda _{t}=0.$$

($\bar c$) Take a positive number $\varepsilon>1$ small enough satisfying $\log_T(1-x)\ge -2x$ for any $x\in [0, (T^{2}+1) \varepsilon )$. For any $\epsilon>0$, write $$L_{\epsilon }(m)= \#(\{ 0\le j\le m-1:\lambda _j<\epsilon \}).$$
By Lemma \ref{lem08} and ($\bar b$), we have $\lim\limits_{k \to \infty}\frac{L_{\varepsilon }(b_k)}{b_k}=1$, which yields \begin{equation}\label{0623}
\lim\limits_{k \to \infty}\big(1-\frac{L_{\varepsilon }(b_k)}{b_k}\big)=0.\end{equation}

If $\lambda _j<\varepsilon $, we have \begin{align*}
0\ge \frac{1}{b_k}\sum_{\substack{j=0\\\lambda _j<\varepsilon }}^{b_k-1} \log_{T}{\Big[1-\big(T^{2}+1 \big)\lambda _j \Big]}
\ge \frac{-2\big(T^{2}+1 \big)}{b_k}\sum_{\substack{j=0\\\lambda _j<\varepsilon  }}^{b_k-1}\lambda _j
\ge -2\big(T^{2}+1 \big)\big(\frac{1}{b_k}\sum_{j=0}^{ b_k-1}\lambda _j \big).\end{align*}
Combining ($\bar b$), we obtain $$\lim_{k \to \infty}\frac{1}{b_k} \sum_{\substack{j=0\\\lambda _j<\varepsilon } }^{b_k-1}\log_T[1-(T^{2}+1)\lambda_j ] =0,$$ which implies \begin{equation}\label{0624}
\lim_{k \to \infty} \Big[\prod_{\substack{j=0\\\lambda _j<\varepsilon } }^{b_k-1}\big(1-(T^{2}+1)\lambda _j \big) \Big]^{\frac{1}{b_k}}=1.\end{equation}

If condition {\rm(A)} of Theorem {\rm\ref{thm2}} holds, then for any $j\ge 1$, each basic interval of $S_{j-1}$ contains at most $T^2$ basic intervals of $S_j$. Thus, we have
\begin{align*} \frac{l(S_{j})}{l(S_{j-1})} \le \min\left \{ 1,T^{2}\overline{\Lambda}_j\right \}. \end{align*}
Since $\{l(S_{j})\}_{j\ge 0}$ is a decreasing sequence, we have \begin{equation}\label{0625} l(S_{b_k})\le \prod_{j\in A}\big(T^{2}\overline{\Lambda}_j \big) \end{equation}
for any set $A\subseteq \left \{ 1,2,\cdots ,b_k \right \}$. Notice that for any $j\ge 1$, \begin{equation}\label{0626}
\tau _{j-1}=\min\left \{ \frac{\sum_{i=1}^{N(\tilde{I}) } |\tilde{I}_i| }{|\tilde{I}|}:\tilde{I}\in \mathcal{S}_{j-1} \right \} \ge
\frac{\min\limits_{\tilde{I}\in\mathcal{S}_{j}}|\tilde{I}|}{\max\limits_{\tilde{I}\in\mathcal{S}_{j-1}}|\tilde{I}|}
=\underline{\Lambda}_{j},\end{equation}
and by (4) of Lemma \ref{lem06},  we obtain \begin{equation}\label{0627}
\overline{\Lambda}_{j}\le 4(1+w_1)^{2}\underline{\Lambda}_{j} \end{equation}
for any $j\ge 1$.
Then by (\ref{0625}), (\ref{0626}) and (\ref{0627}),  we have  \begin{equation}\label{0628} l(S_{b_k})\le \prod_{j\in A}\big(T^{2} \overline{\Lambda}_j \big) \le \prod_{j\in A}\big(4T^{2}(1+w_1)^{2}\underline{\Lambda}_j \big) \le \prod_{j\in A}\big(4T^{2}(1+w_1)^{2}\tau_{j-1}\big).\end{equation}

Combining (\ref{0623}), (\ref{0624}), (\ref{0628}), Lemma \ref{lem07} and ($\bar a$), we obtain that
\begin{align*}
1&\ge \lim_{k \to \infty} \big(\prod_{j=0}^{b_k-1}\tau _j \big)^{\frac{1}{b_k} }\\
&=\lim_{k \to \infty} \big(\prod_{\substack{j=0\\\lambda _j<\varepsilon } } ^{b_k-1}\tau _j \big)^{\frac{1}{b_k} }
\big(\prod_{\substack{j=0\\\lambda _j\ge \varepsilon }  }^{b_k-1}\tau _j \big)^{\frac{1}{b_k} }\\
&\ge \lim_{k \to \infty} \Big[\prod_{\substack{j=0\\\lambda _j<\varepsilon} }^{b_k-1} \big(1-(T^{2}+1)\lambda _j \big) \Big]^{\frac{1}{b_k} } \Big(\prod_{\substack{j=0\\\lambda _j\ge \varepsilon} }^{b_k-1}\frac{1}{4T^{2}(1+w_1)^{2}} \Big)^{\frac{1}{b_k} }\Big(l(S_{b_k})\Big)^{\frac{1}{b_k} }\\
&\\&
\ge \lim_{k \to \infty} \Big[\prod_{\substack{j=0\\\lambda _j<\varepsilon } }^{b_k-1}\big(1-(T^{2}+1)\lambda _j \big) \Big]^{\frac{1}{b_k} }
\Big(\frac{1}{4T^{2}(1+w_1)^{2}} \Big)^{1-\frac{L_{\varepsilon } (b_k)}{b_k} }\Big(l(S_{b_k})\Big)^{\frac{1}{b_k} }\\
&=1.\end{align*}
Therefore, $\lim\limits_{k \to \infty} (\prod_{j=0}^{b_k-1}\tau _j)^{\frac{1}{b_k}} =1$, which implies
$$\lim\limits_{k \to \infty} \frac{1}{b_k}\sum_{j=0}^{b_k-1}\log_{T}{\tau _j}=0.$$

If condition {\rm(B)} of Theorem {\rm\ref{thm2}} holds, then by the similar argument  to the case that the condition {\rm(A)} of Theorem {\rm\ref{thm2}} holds(replace $(1+w_1)$ by $w_2$ ), we can also obtain $\lim\limits_{k \to \infty} \frac{1}{b_k}\sum_{j=0}^{b_k-1}\log_{T}{\tau _j}=0.$

($\bar d$)  If condition {\rm(A)} of Theorem {\rm\ref{thm2}} holds, then by (4) of Lemma \ref{lem06}, we have for any $j\ge 1$, $\tilde{I}\in \mathcal {S}_j$, and $\tilde{J}\in \mathcal{S}_{j-1}$ with $\tilde{I}\subset \tilde{J}$,
\begin{equation}\label{0629}
\pi _j\le \frac{\max\limits_{I^{'}\in\mathcal{S}_j}|I^{'}|}{\min\limits_{J^{'}\in\mathcal{S}_{j-1}}|J^{'}|}
\le \frac{2(1+w_1) \min\limits_{I^{'}\in\mathcal{S}_j}|I^{'}|}{\frac{1}{2(1+w_1)} \max\limits_{J^{'}\in\mathcal{S}_{j-1}}|J^{'}|}
\le \frac{4(1+w_1)^{2}|\tilde{I}|}{|\tilde{J}|}.
\end{equation}

Taking $I^{*}\in \mathcal {S}_j$ and $J^{*}\in \mathcal {S}_{j-1}$ with $I^{*}\subset J^{*}$, such that $\pi _j=\frac{|I^{*}|}{|J^{*}|}$. According to the reconstruction process, $J^{*}$ contains at least 2 basic intervals of $S_j$, suppose that $I^{**}\in \mathcal {S}_j$, $I^{**}\neq I^{*}$ with $I^{**}\subset J^{*}$. Then by (\ref{0629}), we have
\begin{equation*} 1\ge\frac{|I^{*}|+|I^{**}|}{|J^{*}|} =\pi _j+\frac{|I^{**}|}{|J^{*}|}\ge \pi _j+\frac{\pi _j}{4(1+w_1)^{2}}
=\frac{1+4(1+w_1)^{2}}{4(1+w_1)^{2}}\pi _j. \end{equation*}
Let $\beta \in (\frac{4(1+w_1)^{2}}{1+4(1+w_1)^{2}} ,1)$, then we have $$\lim\limits_{k\to \infty}\frac{S_\beta(b_k)}{b_k} =\lim\limits_{k\to \infty}\frac{\#( \{ 1\le j\le b_k:\pi _j<\beta  \}) }{b_k} =1.$$

If condition {\rm(B)} of Theorem {\rm\ref{thm2}} holds, by the similar argument  to the case that the condition {\rm(A)} of Theorem {\rm\ref{thm2}} holds(replace $(1+w_1)$ by $w_2$), we can also obtain
$\lim\limits_{k\to \infty}\frac{S_\beta(b_k)}{b_k}=1$.

\end{proof}

\subsection{The probability measure on the quasisymmetric image sets}
Let $E\in \mathcal{M}(I_{0},\left\{n_{k}\right\},\left\{c_{k}\right\})$  satisfy the conditions of Theorem \ref{thm2}, $\{S_{m}\}_{m\ge0}$ and $\{b_k\}_{k\ge 0}$ be the sequences in Lemma \ref{lem06} and lemma \ref{lem09} respectively, $f$ be a \text{\rm1}-dimensional quasisymmetric mapping. In order to estimate the lower bound of the packing dimension of $f(E)$ by Lemma \ref{lem03}, we need to define a probability Borel measure on the quasisymmetric image set $f(E)$.

For any $m\ge0$ and any  basic interval of $S_m$, which denoted by $I_m$, write $J_{m}=f(I_m)$. Obviously, $f(S_m)$ is the union of the images of the basic intervals of $S_m$ under $f$, and the interiors of the images are disjoint since $f$ is a homeomorphism. Similarly, $J_m$ is called a basic interval of $f(S_m)$, and write $f(\mathcal{S}_{m})=\{B: B ~~~~\text{is a basic interval of}~~~~ f(S_{m})\}$. Write $J_{m,1},J_{m,2}\cdots,J_{m,N(J_{m})}$ for all basic intervals of $f(S_{m+1})$ contained  in $J_{m}$, where $J_{m,1},J_{m,2}\cdots,J_{m,N(J_{m})}$  are listed from left to right, then $N(J_{m})\le T^{2}$ by (3) of Lemma \ref{lem06}.

For any $d\in(0,1)$, $m\ge0$ and $1\le i\le N(J_{m})$, according to the measure extension theorem, we can define $\mu_{d}$, where $\mu_{d}$ is a  probability Borel measure  on the quasisymmetric image set $f(E)$,  fulfilling $\mu _d(f(S_0))=1$ and
\begin{equation}\label{0630}
\mu_d(J_{m,i})=\frac{\left | J_{m,i} \right |^d }{\sum_{j=1}^{N(J_{m})}\left | J_{m,j} \right |^d }\mu_d(J_{m}).
\end{equation}
\bigskip

For any $d\in(0,1)$, $k\ge 1$ and any basic interval of $f(S_{b_k})$, denoted by $U$,  we proceed to estimate  $\mu_d(U)$.

\begin{proposition}\label{pro1} For any $d\in(0,1)$ and $k\ge1$, let $U$ be a basic interval  of $f(S_{b_k})$, then there is a positive constant $C$ $($independent of $U$$)$,  such that $\mu _d(U)\le C|U|^{d}$.
\end{proposition}

\begin{proof}
If condition {\rm(A)} or condition {\rm(B)} of Theorem {\rm\ref{thm2}} holds, then for any $d\in(0,1)$, $k\ge 1$, let $U=J_{b_k}$ be a basic interval of $f(S_{b_k})$. For any $0\le j\le b_k-1$, let $J_j$ be a basic interval of $f(S_j)$ satisfying $$U=J_{b_k}\subset J_{b_k-1}\subset \cdots \subset J_1\subset J_0=f(S_0).$$  Without loss of generality, suppose $J_0=f(S_0)=[0,1]$.
According to the definition of $\mu_d$, we obtain
\begin{align*}
\mu_d(J_{b_k})
&=\frac{|J_{b_k}|^{d}}{\sum_{i=1}^{N(J_{b_k-1})}|J_{b_k-1,i}|^{d} } \mu _d(J_{b_k-1})\\
&=\frac{|J_{b_k}|^{d}}{\sum_{i=1}^{N(J_{b_k-1})}|J_{b_k-1,i}|^{d} } \cdot \frac{|J_{b_k-1}|^{d}}{\sum_{i=1}^{N(J_{b_k-2})}|J_{b_k-2,i}|^{d} }\cdots \frac{|J_1|^{d}}{\sum_{i=1}^{N(J_0)}|J_{0,i}|^{d} } \cdot |J_0|^{d}, \end{align*}
which yields \begin{align}\label{0631}
\frac{\mu _d(J_{b_k})}{|J_{b_k}|^{d}} =\prod_{j=0}^{b_k-1} \frac{|J_j|^{d}}{\sum_{i=1}^{N(J_j)}|J_{j,i}|^{d} }.
\end{align}
Therefore, to prove Proposition \ref{pro1}, we only need to prove \begin{align}\label{0632}
\liminf\limits_{k\to \infty}\big(\prod_{j=0}^{b_k-1} \frac{\sum_{i=1}^{N(J_j)}|J_{j,i}|^{d}}{|J_j|^{d} } \big)^{\frac{1}{b_k}}>1.
\end{align}

We begin to estimate $\frac{\sum_{i=1}^{N(J_j)}|J_{j,i}|^{d}}{|J_j|^{d} }$ for any $0\le j\le b_k-1$. By the argument above, for a basic interval of $f(S_j)$ which is denoted by  $J_j$, $J_{j,1},J_{j,2},\cdots ,J_{j,N(J_j)}$(listed from left to right)  are all of the basic intervals of $f(S_{j+1})$ contained in $J_j$, and $I_j=f^{-1}(J_j)$, where $I_j$ is a basic interval of $S_j$. For any $1\le l\le N(J_j)-1$,  write $$G_{j,0}=\big[\min(J_j),\min(J_{j,1})\big),~G_{j,N(J_j)}=\big(\max(J_{j,N(J_j)}),\max(J_j)\big],$$
$$G_{j,l}=\big(\max(J_{j,l}),\min(J_{j,l+1})\big).$$
Then $J_j=\big(\bigcup_{i=1}^{N(J_j)}J_{j,i} \big)\bigcup \big(\bigcup_{l=0}^{N(J_j)}G_{j,l} \big)$, and we call $G_{j,l}(0\le l\le N(J_j))$ the $(j+1)$-order gap contained in $J_j$. Notice that there may exists $0\le l\le N(J_j)$ satisfying $G_{j,l}=\emptyset$.
Let $I_{j,i}=f^{-1}(J_{j,i})\subseteq S_{j+1}$ for any $1\le i\le N(J_j)$, and let $L_{j,l}=f^{-1}(G_{j,l})\subseteq I_j-S_{j+1}$ for any $0\le l\le N(J_j)$. Obviously, any $I_{j,i}(1\le i\le N(J_j))$ is basic interval of $S_{j+1}$ contained in $I_j$.

Without loss of generality, suppose $|J_{j,1}|=\max\limits _{1\le i\le N(J_j)}\{|J_{j,i}|\}$. Then $0<\frac{|J_{j,i}|}{|J_{j,1}|}\le 1$ for any $1\le i\le N(J_j)$. Since $0<d<1$, we have $(\frac{|J_{j,i}|}{|J_{j,1}|})^{d} \ge \frac{|J_{j,i}|}{|J_{j,1}|}$, which implies \begin{equation}\label{0633}
\begin{aligned}
\frac{\sum_{i=1}^{N(J_j)}|J_{j,i}|^{d} }{\big(\sum_{i=1}^{N(J_j)}|J_{j,i}| \big)^{d}}
&=\frac{1+\big(\frac{|J_{j,2}|}{|J_{j,1}|} \big)^{d}+ \big(\frac{|J_{j,3}|}{|J_{j,1}|} \big)^{d}+\cdots+ \big(\frac{|J_{j,N(J_j)}|}{|J_{j,1}|} \big)^{d}}
{\big(1+\frac{|J_{j,2}|}{|J_{j,1}|}+\frac{|J_{j,3}|}{|J_{j,1}|}+\cdots +\frac{|J_{j,N(J_j)}|}{|J_{j,1}|}\big)^{d}}\\
&\ge \big(1+\frac{|J_{j,2}|}{|J_{j,1}|}+\frac{|J_{j,3}|}{|J_{j,1}|}+\cdots +\frac{|J_{j,N(J_j)}|}{|J_{j,1}|}\big)^{1-d}\\
&\ge 1.
\end{aligned} \end{equation}
Hence, \begin{align}\label{0634}
\frac{\sum_{i=1}^{N(J_j)}|J_{j,i}|^{d} }{|J_{j}|^{d}}=
\frac{\sum_{i=1}^{N(J_j)}|J_{j,i}|^{d} }{\big(\sum_{i=1}^{N(J_j)}|J_{j,i}|\big)^{d}}\cdot
\frac{\big(\sum_{i=1}^{N(J_j)}|J_{j,i}|\big)^{d}}{|J_{j}|^{d}} \ge
\frac{\big(\sum_{i=1}^{N(J_j)}|J_{j,i}|\big)^{d}}{|J_{j}|^{d}}. \end{align}

Let $\beta \in (0,1)$ and  $p \in (0,1]$ be the constant in ($\bar d$) of Lemma \ref{lem09} and  the constant in Lemma \ref{lem04}  respectively, and  $\epsilon \in (0,1)$ be a constant small enough which  satisfies
\begin{enumerate}
\item[\textup{$(1^*)$}]$\epsilon <\frac{1-\beta }{T^{2}+1}$;
\item[\textup{($2^*$)}]$(1-\bar{Q}x^{p})\ge (1-x^{p})^{\bar{Q}+1}$ for any $x\in [0,\epsilon)$, where $\bar{Q}=4(T^{2}+1)$;
\item[\textup{($3^*$)}]$\log_T{(1-x^{p})}\ge-2x^{p}$  for any $x\in [0,\epsilon)$.
\end{enumerate}

For any $k\ge 1$, we use the following notations from Lemma \ref{lem09} and its proof:
$$L_\epsilon (b_k)= \#(\{ 0\le j\le b_k-1:\lambda _j<\epsilon \}),~ S_\beta (b_k)=\#(\{ 1\le j\le b_k:\pi_j<\beta\}),$$
and write
$$M_{\epsilon,\beta}(b_k)=\#( \{1\le j\le b_k:\lambda _{j-1}<\epsilon,~ \pi_j<\beta \}).$$
By Lemma \ref{lem08} and $(\bar b)$ of Lemma \ref{lem09}, we have  \begin{equation}\label{0635} \lim\limits_{k\to \infty}\frac{L_\epsilon (b_k)}{b_k}=1.\end{equation}
By $(\bar d)$ of Lemma \ref{lem09}, we obtain $\lim\limits_{k\to \infty}\frac{S_\beta (b_k)}{b_k}=1$,  which yields \begin{equation}\label{0636}  \lim\limits_{k\to \infty}\frac{M_{\epsilon,\beta}(b_k)}{b_k}=1.\end{equation}

To get (\ref{0632}), we distinguish $\prod_{j=0}^{b_k-1} \frac{\sum_{i=1}^{N(J_j)}|J_{j,i}|^{d}}{|J_j|^{d} }$ into two cases. We write $$\mathrm{I}=\prod_{\substack{j=0\\\lambda _j<\epsilon } }^{b_k-1}
\frac{\sum_{i=1}^{N(J_j)}|J_{j,i}|^{d}}{|J_{j}|^{d}},~~ \mathrm{II}=\prod_{\substack{j=0\\\lambda _j\ge \epsilon } }^{b_k-1}
\frac{\sum_{i=1}^{N(J_j)}|J_{j,i}|^{d}}{|J_{j}|^{d}},$$
then \begin{align}\label{0637}\prod_{j=0}^{b_k-1} \frac{\sum_{i=1}^{N(J_j)}|J_{j,i}|^{d} }{|J_{j}|^{d}} =\mathrm{I}\cdot \mathrm{II}.\end{align}

\textbf{Case 1:} If $\lambda _j<\epsilon$. In this case, we  write $$\mathrm{I}_1=\prod_{\substack{j=0\\\lambda _j<\epsilon,\pi _{j+1}<\beta}}^{b_k-1}\frac{\sum_{i=1}^{N(J_j)}|J_{j,i}|^{d}}{|J_{j}|^{d}},  ~~\mathrm{I}_2=\prod_{\substack{j=0\\\lambda _j<\epsilon,\pi _{j+1}\ge \beta}}^{b_k-1}\frac{\sum_{i=1}^{N(J_j)}|J_{j,i}|^{d}}{|J_{j}|^{d}},$$
then \begin{align}\label{0638}\mathrm{I}=\mathrm{I}_1\cdot \mathrm{I}_2.\end{align}

Notice that $\frac{|L_{j,l}|}{|I_j|} \le \lambda _j$ for any $0\le l\le N(J_j)$. By Lemma \ref{lem04}, $\frac{|f(L_{j,l})|}{|f(I_j)|} \le 4\big(\frac{|L_{j,l}|}{|I_j|}\big)^{p}$,  which implies $$\frac{|G_{j,l}|}{|J_j|} \le 4\lambda _{j}^{p}$$ for any $0\le l\le N(J_j)$.
Since $N(J_j)\le T^{2}$,  by $\lambda _j<\epsilon$ and ($2^*$), we have \begin{equation}
\begin{aligned}\label{0639}
\frac{\big(\sum_{i=1}^{N(J_j)}|J_{j,i}|\big)^{d}}{|J_{j}|^{d}}
&=\big(\frac{|J_j|-\sum_{l=0}^{N(J_j)}|G_{j,l}| }{|J_j|} \big)^{d}\\
&\ge\big(1-4(T^{2}+1)\lambda _{j}^{p}\big)^{d}\\
&\ge\Big[\big(1-\lambda _{j}^{p}\big)^{4(T^{2}+1)+1}\Big]^{d}\\
&=\big(1-\lambda _{j}^{p}\big)^{(\bar{Q}+1)d}.
\end{aligned} \end{equation}
Moreover, if $\lambda _j<\epsilon $ and $\pi _{j+1}<\beta$, since $q\ge 1$ and $N(J_j)\le T^{2}$, by Lemma \ref{lem04} and the Jensen inequality, we have \begin{align}\label{0640}
\frac{\sum_{l=2}^{N(J_j)}|J_{j,l}|}{|J_{j}|}\ge \lambda \cdot
\frac{\sum_{l=2}^{N(J_j)}|I_{j,l}|^{q}}{|I_{j}|^{q}}\ge
\big(T^{2}-1\big)^{1-q}\cdot \lambda \big(\frac{\sum_{l=2}^{N(J_j)}|I_{j,l}|}{|I_{j}|}\big)^{q}. \end{align}
Notice that $\frac{|I_{j,1}|}{|I_j|}\le\pi _{j+1}<\beta$, and $\frac{|L_{j,l}|}{|I_j|}\le \lambda _j<\epsilon $ for any $0\le l\le N(J_j)\le T^{2}$, it follows that \begin{align}\label{0641}
\frac{\sum_{l=2}^{N(J_j)}|I_{j,l}|}{|I_{j}|}=\frac{|I_j|-|I_{j,1}|-\sum_{l=0}^{N(J_j)}|L_{j,l}|}{|I_j|}
\ge 1-\beta -(T^{2}+1)\epsilon. \end{align}
Combining (\ref{0640}) and (\ref{0641}), we have \begin{align}\label{0642} \frac{\sum_{l=2}^{N(J_j)}|J_{j,l}|}{|J_{j}|}\ge \big(T^{2}-1\big)^{1-q}\cdot \lambda \cdot \big(1-\beta -(T^{2}+1)\epsilon \big)^{q}. \end{align}

By Lemma \ref{lem04}, we have \begin{align}\label{0643}
\frac{|J_{j,1}|}{|J_j|} =\frac{|f(I_{j,1})|}{|f(I_j)|}\le 4\big(\frac{|I_{j,1}|}{|I_j|}\big)^{p}\le 4\beta ^{p}. \end{align}
Combining (\ref{0642}) and (\ref{0643}), we obtain \begin{align*}
\frac{\sum_{l=2}^{N(J_j)}|J_{j,l}|}{|J_{j,1}|}=
\frac{\sum_{l=2}^{N(J_j)}|J_{j,l}|}{|J_j|}\cdot\frac{|J_j|}{|J_{j,1}|}&\ge
\frac{|J_j|}{|J_{j,1}|} \cdot \frac{\lambda \big(1-\beta -(T^{2}+1)\epsilon \big)^{q}}{\big(T^{2}-1 \big)^{q-1}}\\ &\ge
\frac{\lambda \big(1-\beta -(T^{2}+1)\epsilon \big)^{q}}{4\beta ^{p}\big(T^{2}-1 \big)^{q-1}}. \end{align*}
Notice that $0<d<1$, by (\ref{0633}), we have  \begin{equation}\label{0644}
\begin{aligned}
\frac{\sum_{i=1}^{N(J_j)}|J_{j,i}|^{d} }{\big(\sum_{i=1}^{N(J_j)}|J_{j,i}| \big)^{d}}
&\ge \big(1+\frac{\sum_{l=2}^{N(J_j)}|J_{j,l}|}{|J_{j,1}|} \big)^{1-d}\\
&\ge \Big[1+\frac{\lambda \big(1-\beta -(T^{2}+1)\epsilon \big)^{q}}{4\beta ^{p}\big(T^{2}-1 \big)^{q-1} } \Big]^{1-d}.
\end{aligned}\end{equation}
Write $ \Big[1+\frac{\lambda \big(1-\beta -(T^{2}+1)\epsilon \big)^{q}}{4\beta ^{p}\big(T^{2}-1 \big)^{q-1} } \Big]^{1-d} \triangleq R$, then $R>1$.

Combining (\ref{0634}) and (\ref{0639}), if $\lambda _j<\epsilon $, we have  \begin{align}\label{0645}
\frac{\sum_{i=1}^{N(J_j)}|J_{j,i}|^{d} }{|J_{j}|^{d}}\ge \big(1-\lambda _j^{p} \big)^{(\bar{Q}+1)d}.\end{align}
Combining (\ref{0639}) and (\ref{0644}), if $\lambda _j<\epsilon $ and $\pi _{j+1}<\beta$, we have \begin{align}\label{0646}
\frac{\sum_{i=1}^{N(J_j)}|J_{j,i}|^{d} }{|J_{j}|^{d}}=
\frac{\sum_{i=1}^{N(J_j)}|J_{j,i}|^{d} }{\big(\sum_{i=1}^{N(J_j)}|J_{j,i}| \big)^{d}}\cdot
\frac{\big(\sum_{i=1}^{N(J_j)}|J_{j,i}| \big)^{d}}{|J_{j}|^{d}}\ge R\big(1-\lambda _j^{p} \big)^{(\bar{Q}+1)d}. \end{align}
Then by (\ref{0638}), (\ref{0645}) and (\ref{0646}), we obtain \begin{equation}\label{0647}
\begin{aligned}
\mathrm{I}&\ge \prod_{\substack{j=0\\\lambda _j<\epsilon,\pi _{j+1}< \beta}}^{b_k-1}  R\big(1-\lambda _j^{p} \big)^{(\bar{Q}+1)d}\cdot \prod_{\substack{j=0\\\lambda _j<\epsilon,\pi _{j+1}\ge \beta}}^{b_k-1} \big(1-\lambda _j^{p} \big)^{(\bar{Q}+1)d} \\&= R^{M_{\epsilon,\beta}(b_k)}\cdot \prod_{\substack{j=0\\\lambda _j<\epsilon } }^{b_k-1}\big(1-\lambda _j^{p} \big)^{(\bar{Q}+1)d}. \end{aligned} \end{equation}

On the other hand, for $\lambda _j<\epsilon$, notice that $0<p\le1$,  by ($3^*$) and the Jensen inequality, we have \begin{align*}
0\ge \frac{1}{b_k}\sum_{\substack{j=0\\\lambda _j<\epsilon } }^{b_k-1}\log_{T}{\big(1-\lambda _j^{p} \big)}
\ge -\frac{2}{b_k}\sum_{\substack{j=0\\\lambda _j<\epsilon } }^{b_k-1}\lambda _j^{p}
\ge -\frac{2}{b_k}\sum_{j=0}^{b_k-1}\lambda _j^{p}
\ge -2\big(\frac{1}{b_k}\sum_{j=0}^{b_k-1}\lambda _j \big)^{p} \end{align*}
for any $k\ge 1$. Notice that $\lim\limits_{k \to \infty} -2\big(\frac{1}{b_k}\sum_{j=0}^{b_k-1}\lambda _j \big)^{p}=0$ by $(\bar b)$ of Lemma \ref{lem09},  then  we obtain $\lim_{k \to \infty}\frac{1}{b_k}\sum_{\substack{j=0\\\lambda _j<\epsilon } }^{b_k-1}\log_{T}{\big(1-\lambda _j^{p} \big)}=0$, which yields \begin{align}\label{0648}
\lim_{k \to \infty}\big[\prod_{\substack{j=0\\\lambda _j<\epsilon } }^{b_k-1}(1-\lambda _j^{p}) \big]^{\frac{1}{b_k} } =1. \end{align}

By (\ref{0636}), (\ref{0647}) and (\ref{0648}), we obtain \begin{align}\label{0649}
\liminf \limits _{k\to \infty }\mathrm{(I)}^{\frac{1}{b_k}}\ge \liminf \limits _{k\to \infty }R^{\frac{M_{\epsilon,\beta}(b_k)}{b_k}}\cdot
\liminf \limits _{k\to \infty }\prod_{\substack{j=0\\\lambda _j<\epsilon }}^{b_k-1}\big(1-\lambda _j^{p} \big)^{\frac{(\bar{Q}+1)d}{b_k}}= R.  \end{align}

\textbf{Case 2:} If $\lambda _j \ge \epsilon$. In this case, by Lemma \ref{lem04} and the Jensen inequality, notice that $q\ge 1$ and $N(J_j)\le T^{2}$, we obtain \begin{align*}
\frac{\sum_{i=1}^{N(J_j)}|J_{j,i}| }{|J_{j}|}\ge \lambda \cdot \frac{\sum_{i=1}^{N(J_j)}|I_{j,i}|^{q}}{|I_{j}|^{q}} \ge
\frac{\lambda }{(T^{2})^{q-1}} \big(\frac{\sum_{i=1}^{N(J_j)}|I_{j,i}|}{|I_{j}|} \big)^{q}\ge
\frac{\lambda }{(T^{2})^{q-1}}\tau _j^{q}.\end{align*}
By (\ref{0634}), \begin{align}\label{0650}
\frac{\sum_{i=1}^{N(J_j)}|J_{j,i}|^{d} }{|J_{j}|^{d}}\ge \frac{\big(\sum_{i=1}^{N(J_j)}|J_{j,i}| \big)^{d} }{|J_{j}|^{d}}\ge
\big[\frac{\lambda }{(T^{2})^{q-1}}\tau _j^{q} \big]^{d}. \end{align}
Hence, \begin{align*}
\mathrm{II}=\prod_{\substack{j=0\\\lambda _j\ge \epsilon } }^{b_k-1}
\frac{\sum_{i=1}^{N(J_j)}|J_{j,i}|^{d}}{|J_{j}|^{d}}
\ge \prod_{\substack{j=0\\\lambda _j\ge \epsilon } }^{b_k-1}\big[\frac{\lambda }{(T^{2})^{q-1}}\tau _j^{q} \big]^{d}
&\ge (\prod_{j=0}^{b_k-1} \tau _j)^{qd} \cdot \prod_{\substack{j=0\\\lambda _j\ge \epsilon } }^{b_k-1}(\frac{\lambda}{(T^{2})^{q-1}})^{d}\\
&=(\prod_{j=0}^{b_k-1} \tau _j)^{qd} \cdot
(\frac{\lambda}{(T^{2})^{q-1}})^{d(b_k-L_\epsilon (b_k))}.
\end{align*}
Therefore, by $(\bar c)$ of Lemma \ref{lem09} and (\ref{0635}), we have   \begin{align}\label{0651}
\liminf \limits _{k\to \infty }\mathrm{(II)}^{\frac{1}{b_k}}\ge \liminf_{k \to \infty} (\prod_{j=0}^{b_k-1} \tau _j)^{\frac{qd}{b_k} }\cdot\liminf \limits _{k\to \infty }(\frac{\lambda}{(T^{2})^{q-1}})^{\frac{d(b_k-L_\epsilon (b_k) )}{b_k} }=1. \end{align}

Combining (\ref{0637}), (\ref{0649}) and (\ref{0651}), we have  \begin{align*}
\liminf \limits _{k\to \infty}\big(\prod_{j=0}^{b_k-1} \frac{\sum_{i=1}^{N(J_j)}|J_{j,i}|^{d} }{|J_{j}|^{d}} \big)^{\frac{1}{b_k}}\ge \liminf \limits _{k\to \infty}(\mathrm{I})^{\frac{1}{b_k}}\cdot \liminf \limits _{k\to \infty}(\mathrm{II})^{\frac{1}{b_k}}\ge R\cdot 1>1, \end{align*}
which means (\ref{0632}) holds.

Therefore, there exists a constant $C>0$ satisfying  \begin{equation*}
\frac{\mu_d(J_{b_k})}{\left | J_{b_k} \right |^d }=\prod_{j=0}^{b_k-1}\frac{\left | J_j \right |^d}{\sum_{i=1}^{N(J_j)}\left | J_{j,i} \right |^d}\le C,\end{equation*}
which implies that we finish the proof of  Proposition \ref{pro1}.
\end{proof}

\subsection{The proof of Theorem \ref{thm2} }
We do the final proof of Theorem \ref{thm2}. Let $E\in \mathcal{M}(I_{0},\left\{n_{k}\right\},\left\{c_{k}\right\})$ satisfy the conditions of Theorem \ref{thm2} with $\dim_PE=1$, $f$ be a \text{\rm1}-dimensional quasisymmetric mapping, $\{b_k\}_{k\ge 0}$ be the sequence in Lemma \ref{lem09}. For any $y\in f(E)$, since $f$ is a homeomorphism, the function given by $g_y(r)=|f^{-1}\big(B(y,r) \big)|$ is monotone increasing, satisfying $\lim\limits_{r \to 0} g_y(r)=0$. Then there exists a sequence $\left \{ r_k \right \} _{k\ge 1}$ with $\lim\limits_{k \to \infty} r_k=0$, fulfilling \begin{align}\label{0652}
\min_{I\in \mathcal{S}_{b_k}}|I|\le g_y(r_k)<\min_{I\in \mathcal{S}_{b_k-1}}|I|. \end{align}

If condition {\rm(A)} of Theorem {\rm\ref{thm2}} holds, then for any $k\ge 1$, $$\#\big (\{ I:I\in \mathcal{S}_{b_k-1},I\cap f^{-1}(B(y,r_k))\ne \emptyset  \}\big )\le 2.$$
Thus, $$\#\big (\{ I:I\in \mathcal{S}_{b_k},I\cap f^{-1}(B(y,r_k))\ne \emptyset  \}\big )\le 2T^2,$$
which implies $$\#\big (\{ J:J\in f(\mathcal{S}_{b_k}),J\cap B(y,r_k)\ne \emptyset  \}\big )\le 2T^2.$$
Suppose $G_1,G_2, \cdots , G_l(1\le l\le 2T^{2})$ are all of the basic intervals of $f(S_{b_k})$ which intersect $B(y,r_k)$, then $B(y,r_k)\cap f(E)\subset \big(\bigcup_{i=1}^{l}G_i \big)$.
By Proposition \ref{pro1}, we obtain  \begin{align}\label{0653}
\mu _d\big(B(y,r_k) \big)=\mu _d\big(B(y,r_k)\cap f(E) \big)\le \mu _d\big(\bigcup_{i=1}^{l}G_i \big)\le
\sum_{i=1}^{l} \mu _d(G_i)\le C\sum_{i=1}^{l}|G_i|^{d}. \end{align}

By (\ref{0652}) and (4) of Lemma \ref{lem06}, we have  \begin{align}\label{0654}
\min_{I\in \mathcal{S}_{b_k}}|I|\le |f^{-1}\big(B(y,r_k) \big)|,~\max_{I\in \mathcal{S}_{b_k}}|I|\le 2\big(1+w_1 \big)\min_{I\in \mathcal{S}_{b_k}}|I|.  \end{align}
Thus, for any $1\le i\le l$,  $$|f^{-1}(G_i)|\le \max_{I\in \mathcal{S}_{b_k}}|I|\le 2\big(1+w_1 \big)\min_{I\in \mathcal{S}_{b_k}}|I|\le 2\big(1+w_1 \big)|f^{-1}\big(B(y,r_k) \big)|.$$
Note that $B(y,r_k)\bigcap G_i\ne \emptyset$, then we have  \begin{align}\label{0655}
f^{-1}(G_i)\subset 6\big(1+w_1 \big)f^{-1}\big(B(y,r_k) \big), \end{align}
where $\rho I$ is the interval concentric with $I$ of length  $\rho |I|$ for any interval $I$ and $\rho >0$.

By (\ref{0655}), Lemma \ref{lem04} and the homeomorphism of $f$, we obtain \begin{align}\label{0656}
|G_i|\le |f\Big(6\big(1+w_1 \big)f^{-1}\big(B(y,r_k) \big) \Big)|\le  K_{6(1+w_1)}|B(y,r_k)|= 2K_{6(1+w_1)}r_k. \end{align}
Notice that $1\le l \le 2T^2$ and $d\in (0,1)$, combining (\ref{0653}) and (\ref{0656}), we have  \begin{align*}
\mu _d\big(B(y,r_k) \big)&\le C\sum_{i=1}^{l}|G_i|^{d}\le C\cdot 2T^{2}\big(2K_{6(1+w_1)}r_k \big)^{d}\le 4K_{6(1+w_1)}^{d}T^{2}C\big(r_k\big)^{d}.  \end{align*}
Write $ 4K_{6(1+w_1)}^{d}T^{2}C \triangleq \tilde{C}_1$, notice that $\lim\limits_{k \to \infty} r_k=0$, then we have
\begin{align}\label{0657} \liminf\limits _{r\to 0}\frac{\mu _d\big(B(y,r) \big)}{r^{d} } \le \tilde{C}_1. \end{align}

By the arbitrariness of  $y\in f(E)$, we obtain $\dim_Pf(E)\ge d$ by (\ref{0657}) and Lemma \ref{lem03}. Then by the arbitrariness of $d\in (0,1)$, we have $\dim_Pf(E)\ge 1$. Obviously, $\dim_Pf(E)\le 1$, which implies $\dim_Pf(E)=1$.

If condition {\rm(B)} of Theorem {\rm\ref{thm2}} holds, by the similar argument  to the case that the condition {\rm(A)} of Theorem {\rm\ref{thm2}} holds(replace $1+w_1$ by $w_2$), we can prove that there exists a constant $\tilde{C}_2>0$ such that
\begin{align}\label{0658} \liminf\limits _{r\to 0}\frac{\mu _d\big(B(y,r) \big)}{r^{d} } \le \tilde{C}_2. \end{align}
By the same argument above, we obtain $\dim_Pf(E)=1$.

Hence we complete  the proof of Theorem {\rm\ref{thm2}}.

\bigskip

\section{Examples}
We state two examples in this section to show that  Theorems \ref{thm1} and \ref{thm2} respectively generalize the results in \cite{wxywuj08} and \cite{lly25}, and state a example to show that we can not improve the conditions of  Theorem \ref{thm1}.

In the next example, we give a  homogeneous Moran set $E\in \mathcal{M}(I_{0},\left\{n_{k}\right\},\left\{c_{k}\right\})$ which satisfies the conditions of   Theorem \ref{thm1} of this paper, but does not satisfy the conditions of Theorem 1.4 of \cite{wxywuj08}.
\begin{exa}\label{exa1}
{\rm
Let $E\in \mathcal{M}(I_{0},\left\{n_{k}\right\},\left\{c_{k}\right\})$ satisfy for any $k\ge 1$ and $\sigma \in D_{k-1}$,
\begin{enumerate}
\item[\textup{$(1)$}] $n_k \equiv 3$, $c_k \equiv \frac{1}{5}$;
\item[\textup{($2$)}] $\xi_{\sigma ,0}=\xi_{\sigma ,3}=0$,  which implies $L_k=R_k=0$.
\end{enumerate}
 Then the number of the $k$-order basic intervals of $E$ is $N_k=\prod_{i=1}^{k}n_i=3^{k} $, and let $I_{\sigma_1},I_{\sigma_2},\cdots ,I_{\sigma_{3^k}}$ denote all of the $k$-order basic intervals of $E$ locating from left to right, where $\sigma_i \in D_k$ for any $1\le i\le 3^k$. Furthermore, $E$ satisfies the following formulas:
$$\xi_{\sigma_1,1}=\frac{1}{5^k}\cdot \frac{3}{10}=\delta_k\cdot \frac{3}{10} ;$$
 $$\xi_{\sigma_1,2}=\frac{1}{5^k}\cdot \frac{1}{10}=\delta_k\cdot \frac{1}{10} ;$$
 $$\xi_{\sigma_i,1}=\frac{1}{5^k}\cdot \frac{1}{10}=\delta_k\cdot \frac{1}{10}(2\le i\le 3^k);$$
$$\xi_{\sigma_i,2}=\frac{1}{5^k}\cdot \frac{3}{10}=\delta_k\cdot \frac{3}{10}(2\le i\le 3^k).$$
Then for any $k\ge 0$, $\sigma \in D_k$,
$$|I_{\sigma }|=\sum_{j=0}^{n_{k+1}}\xi_{\sigma,j}+n_{k+1}\delta_{k+1} =(\frac{1}{10}+\frac{3}{10} )\delta_{k}+3\cdot \frac{1}{5}\delta_{k}=\delta_{k}.$$

It is easy to obtain that $\overline{\nu } _{k+1}=\frac{3}{10}\delta_{k}$, $\underline{\nu } _{k+1}=\frac{1}{10}\delta_{k}$, $\delta_{k+1}=\frac{1}{5} \delta_{k}$ for any $k\ge 0$. Take $w_1=w_2=3$, $w_3=\frac{3}{10}$, it is obvious that $E$ satisfies all of the conditions of Theorem \ref{thm1} of this paper.
Then by (\ref{031}), we obtain $$\overline{\dim}_{B}E=\dim_PE=\limsup\limits_{k\to \infty }\frac{\log{3^{k+1} } }{-\log{5^{-k}}+\log{3}} =\frac{\log{3}}{\log{5}}.$$

However, for any $k\ge 1$, there exist $\sigma, \sigma^{\prime} \in D_k$, such that $\xi_{\sigma,1}\ne \xi_{\sigma^{\prime },1}$  and $\xi_{\sigma,2}\ne \xi_{\sigma^{\prime },2}$, which  means that $E$ is not a homogeneous perfect set, then $E$ does not satisfy the conditions of Theorem 1.4 of \cite{wxywuj08}.

}
\end{exa}

By Remark \ref{rm4}, Any homogeneous Moran set $E$ satisfying the conditions of Theorem 1.4 of \cite{wxywuj08}(which implies that $E$ is a homogeneous perfect set) is a homogeneous Moran set satisfying the conditions of Theorem \ref{thm1} of this paper, equation (1.4) of Theorem 1.4 of \cite{wxywuj08} and equation (\ref{031}) of Theorem \ref{thm1} of this paper are equivalently.  In Example \ref{exa1}, we give a  homogeneous Moran set $E$ such that all of the conditions of   Theorem \ref{thm1} of this paper are fulfilled, but the conditions of Theorem 1.4 of \cite{wxywuj08} are not fulfilled(since $E$ is not a homogeneous perfect set), and calculate the packing and upper box dimensions of $E$ by equation (\ref{031}) of Theorem \ref{thm1} of this paper. Therefore, Theorem \ref{thm1} of this paper is a  generalization of Theorem 1.4 of \cite{wxywuj08}.

\bigskip
In the next example, we give a  homogeneous Moran set $E\in \mathcal{M}(I_{0},\left\{n_{k}\right\},\left\{c_{k}\right\})$ which satisfies the conditions of   Theorem \ref{thm2} in this paper, but does not satisfy the conditions of Theorem 1 of \cite{lly25}.
\begin{exa}\label{exa2}{\rm
Let $E\in \mathcal{M}(I_{0},\left\{n_{k}\right\},\left\{c_{k}\right\})$ satisfy the following conditions:
\begin{enumerate}
\item[\textup{$(1)$}] $n_k =2^{k+1}$, $c_k =(2n_k)^{-1}$ for any $k\ge 1$;
\item[\textup{($2$)}] $\xi_{\sigma ,0}=\xi_{\sigma ,n_{k+1}}=0$ for any $k\ge 0$, $\sigma \in D_k$,  which implies  $L_{k+1}=R_{k+1}=0$.
\end{enumerate}
Then for any $k\ge 1$, the number of the $k$-order basic intervals of $E$ is $$N_k=n_1n_2\cdots n_k=2^{\frac{k(3+k)}{2}}.$$ and let $I_{\sigma_1},I_{\sigma_2},\cdots ,I_{\sigma_{N_k}}$ denote all of the $k$-order basic intervals of $E$ locating from left to right, where $\sigma_i \in D_k$ for any $1\le i\le N_k$. Furthermore, $E$ satisfies the following formulas:
$$\xi_{\sigma_1,1}=2\delta_{k+1};$$ $$\quad \xi_{\sigma_1,2}=\xi_{\sigma_1,3}=\cdots =\xi_{\sigma_1,n_{k+1}-1}=\delta_{k+1};$$
$$\xi_{\sigma_i,n_{k+1}-1 }=2\delta_{k+1}~ (2\le i\le N_k);$$ $$\quad \xi_{\sigma_i,1}=\xi_{\sigma_i,2}=\cdots =\xi_{\sigma_i,n_{k+1}-2}=\delta_{k+1} ~ (2\le i\le N_k).$$
Then for any $k\ge 0$, $\sigma\in D_k$, $\delta_{k+1}=c_{k+1}\delta_{k}=(2n_{k+1})^{-1}\delta_{k}$, and
$$|I_{\sigma }|=\sum_{j=0}^{n_{k+1}}\xi_{\sigma,j}+n_{k+1}\delta_{k+1}
=[(n_{k+1}-2)\delta_{k+1}+2\delta_{k+1}]+n_{k+1}\delta_{k+1}=2n_{k+1}\delta_{k+1}=\delta_{k}.$$

It is easy to obtain that $\overline{\nu } _{k+1}=2\delta_{k+1}$, $\underline{\nu } _{k+1}=\delta_{k+1}$ for any $k\ge 0$. Take $w_1=w_2=2$, it is obvious that $E$ satisfies all of the conditions of Theorem \ref{thm2} in this paper.

Notice that $E$ also satisfies the conditions  (A) and (B) in Theorem \ref{thm1} of this paper, and
$$n_1n_2\cdots n_kn_{k+1}=2^{\frac{(k+1)(k+4)}{2} },$$
$$\delta_{k}=c_1c_2\cdots c_k=(2^{k}n_1n_2\cdots n_k )^{-1}=(2^{k}\cdot 2^{2}\cdot 2^{3}\cdots 2^{k+1})^{-1}=2^{-(\frac{k(5+k)}{2})} .$$
Then by (\ref{031}), we obtain
\begin{align*}
\dim_PE&=\limsup\limits_{k \to\infty}  \frac{\log n_1n_2\cdots n_kn_{k+1}}{-\log(\delta_k-L_{k+1}-R_{k+1})+\log n_{k+1}}\\
&=\limsup\limits_{k \to\infty}  \frac{\log n_1n_2\cdots n_kn_{k+1}}{-\log{\delta_k}+\log n_{k+1}}\\
&=\limsup\limits_{k \to\infty}  \frac{\log{2^{\frac{(k+1)(k+4)}{2} } } }{-\log{2^{-(\frac{k(5+k)}{2})} }+\log{2^{k+2} }  }\\
&=1,
\end{align*}
which implies $E$ is a homogeneous Moran set  which satisfies  all of the conditions of Theorem \ref{thm2} of this paper with $\dim_PE=1$.
Therefore, $E$ is a quasisymmetrically packing-minimal set by Theorem \ref{thm2} of this paper.

However, for any $k\ge 0$, there exist $\sigma, \sigma^{\prime} \in D_k$, such that $\xi_{\sigma,1}\ne \xi_{\sigma^{\prime },1}$  and $\xi_{\sigma,n_{k+1}-1 }\ne \xi_{\sigma^{\prime },n_{k+1}-1 }$, which means that $E$ is not a homogeneous perfect set, then $E$ does not satisfy the conditions of Theorem 1 of \cite{lly25}.

}
\end{exa}
By Remark \ref{rm5},   any homogeneous Moran set $E$ satisfying the conditions in Theorem 1 of \cite{lly25}(which implies that $E$ is a homogeneous perfect set)   is a homogeneous Moran set  satisfying condition (B) of Theorem \ref{thm2} of this paper.  In Example \ref{exa2}, we give a  homogeneous Moran set $E$ such that all of the conditions of   Theorem \ref{thm2} of this paper are fulfilled, but  the conditions of Theorem 1 of \cite{lly25}  are not fulfilled(since $E$ is not a homogeneous perfect set), and obtain that $E$ is quasisymmetrically packing-minimal set by Theorem \ref{thm2} of this paper. Therefore, Theorem \ref{thm2} of this paper is a  generalization of Theorem 1 of \cite{lly25}.

\bigskip
Finally, we give a example to show  that the conclusion of  Theorem \ref{thm1} may not hold if some conditions of  Theorem \ref{thm1}  are lacked.
\begin{exa}\label{exa3}{\rm
Let $E\in \mathcal{M}(I_{0},\left\{n_{k}\right\},\left\{c_{k}\right\})$ satisfy the following conditions:
\begin{enumerate}
\item[\textup{$(1)$}] $n_1 =2$, $c_1 =\frac{1}{4}$ ;
\item[\textup{($2$)}] $n_k=2^{n_1n_2\cdots n_{k-1} }$, $c_k=n_{k}^{-2}$ for any $k\ge 2$;
\item[\textup{($3$)}] $\xi_{\sigma ,0}=\xi_{\sigma ,1}=\cdots =\xi_{\sigma ,n_{k+1}-2}=\xi_{\sigma ,n_{k+1} }=0$(which means $L_{k+1}=R_{k+1}=0$ and $\xi_{\sigma ,n_{k+1}-1}=\delta_k-n_{k+1}\delta_{k+1}$) for any $k\ge 0$, $\sigma \in D_k$.
\end{enumerate}
In this case, it is easy to obtain that conditions (A), (B) and (C) in Theorem \ref{thm1} of this paper do not hold.

For any $\frac{1}{2}<s<1$ and $k_0\ge 1$,
\begin{align*}
\sum_{k=k_0}^{\infty} \sum_{\substack{\sigma\in D_k\\0\le i\le n_{k+1}} }\xi_{\sigma,i}^{s}
&=\sum_{k=k_0}^{\infty}\sum_{\sigma\in D_k }\xi_{\sigma,n_{k+1}-1}^{s} \\
&=\sum_{k=k_0}^{\infty}n_1n_2\cdots n_k(\delta_{k}-n_{k+1}\delta_{k+1})^{s} \\
&\le\sum_{k=k_0}^{\infty}n_1n_2\cdots n_k(\delta _k)^{s} \\
&=\sum_{k=k_0}^{\infty}(n_1n_2\cdots n_k)^{1-2s} \\
&<\infty .\end{align*}
Reference \cite{fk1990} showed that $\dim_PE \le \overline{\dim}_BE$ for any nonempty bounded set $E\subseteq \mathbb{R}$, combing with Lemma \ref{lem01} and the arbitrariness of $s\in (\frac{1}{2}, 1)$, we obtain \begin{equation}\label{bpp} \dim_PE \le \overline{\dim}_BE \le \frac{1}{2}. \end{equation}

On the other hand, for any $k\ge 0$, $\sigma \in D_k$, since $L_{k+1}=R_{k+1}=0$, we have $$\delta_{k}-L_{k+1}-R_{k+1}=\delta_{k}=\prod_{i=1}^{k}c_i=(n_1n_2\cdots n_k)^{-2}.$$
Hence,
\begin{align*}
&\limsup\limits_{k \to\infty}  \frac{\log n_1n_2\cdots n_kn_{k+1}}{-\log(\delta_k-L_{k+1}-R_{k+1})+\log n_{k+1}}\\
&=\limsup\limits_{k \to\infty}  \frac{\log n_1n_2\cdots n_kn_{k+1}}{-\log(n_1n_2\cdots n_k)^{-2}+\log n_{k+1} }\\
&=\limsup\limits_{k \to\infty}  \frac{\log {n_1n_2\cdots n_k} + n_1n_2\cdots n_k\log 2}{2\log{n_1n_2\cdots n_k} +n_1n_2\cdots n_k\log 2 }\\
&=\limsup\limits_{k \to\infty} \frac{ \frac{\log n_1n_2\cdots n_k}{n_1n_2\cdots n_k}  + \log 2}{2\cdot \frac{\log n_1n_2\cdots n_k}{n_1n_2\cdots n_k} + \log 2}\\
&=1.\end{align*}
Combing (\ref{bpp}), we have $$\limsup\limits_{k \to\infty}  \frac{\log n_1n_2\cdots n_kn_{k+1}}{-\log(\delta_k-L_{k+1}-R_{k+1})+\log n_{k+1}}> \overline{\dim}_BE\ge \dim_PE.$$ Therefore, we can not calculate the packing dimension and the upper box dimension  of $E$ by (\ref{031}).

}
\end{exa}

\bigskip
\bigskip

\textbf{Acknowledgement} The authors thank the reviewers for their helpful comments and suggestions.

\bigskip



\bigskip
\begin{thebibliography}{99}
\bibitem{fdj97} Feng D J, Wen Z Y, Wu J. \newblock Some dimensional results for homogeneous Moran sets. Science in China Series A: Mathematics, 1997, 40(5): 475-482.

\bibitem{wzywuj05} Wen Z Y, Wu J. \newblock Hausdorff dimension of homogeneous perfect sets. Acta Mathematica Hungarica, 2005, 107(1): 35-44.

\bibitem{wxywuj08} Wang X Y, Wu J. \newblock Packing dimensions of homogeneous perfect sets. Acta Mathematica Hungarica, 2008, 118(1-2): 29-39.

\bibitem{gwvj73} Gehring W, Vaisala J. \newblock Hausdorff dimension and quasiconformal mappings. Journal of the London Mathematical Society, 1973, 6: 504-512.

\bibitem{gw73} Gehring W. \newblock The $L^p$-integrability of the partial derivatives of a quasiconformal mapping. Bulletin of the American Mathematical Society, 1973, 79: 465-466.

\bibitem{bj99} Bishop J. \newblock Quasiconformal mappings which increase dimension. Annales Academiae Scientiarum Fennicae-mathematica, 1999, 24(2): 397-407.

\bibitem{tj00} Tyson J. \newblock Sets of minimal Hausdorff dimension for quasiconformal maps. Proceedings of the American Mathematical Society, 2000, 128(11): 3361-3367.

\bibitem{kv06} Kovalev V. \newblock Conformal dimension does not assume values between zero and one. Duke Mathematical Journal, 2006, 134(1): 1-13.

\bibitem{hh06} Hakobyan H. \newblock Cantor sets that are minimal for quasisymmetric mappings. Journal of Contemporary Mathematical Analysis, 2006, 41(2): 13-21.



\bibitem{dyx11} Dai Y X, Wen Z X, Xi L F, et al. \newblock Quasisymmetrically minimal Moran sets and Hausdorff dimension. Annales Academiae Scientiarum Fennicae-mathematica, 2011, 36: 139-151.

\bibitem{ww14} Wang W, Wen S Y. \newblock Quasisymmetric minimality of Cantor sets. Topology and its Applications, 2014, 178: 300-314.

\bibitem{yjj18} Yang J J, Wu M, Li Y Z. \newblock On quasisymmetric minimality of homogeneous perfect sets. Fractals, 2018, 26(1): 1850010.


\bibitem{Ahl06} Ahlfors V. \newblock Lectures on quasiconformal mappings. 2nd ed, Maryland: Vol. 38 of Unversity Lecture Series, American Mathematical Society, 2006.

\bibitem{lwx13} Li Y Z, Wu M, Xi L F. \newblock Quasisymmetric minimality on packing dimension for Moran sets. Journal of Mathematical Analysis and Applications, 2013, 408(1): 324-334.



\bibitem{lly25} Liu S S, Li Y Z, Yang J J. \newblock Quasisymmetric minimality on packing dimension for homogeneous perfect sets. Axioms, 2025, 14(10): 751.

\bibitem{hua00} Hua S, Rao H, Wen Z Y,  et al. \newblock On the structures and dimensions of Moran sets. Science in China Series A: Mathematics, 2000, 43(8): 836-852.

\bibitem{tc1981} Tricot C. \newblock Douze d\'{e}finitions de la densit\'{e} logarithmique. Comptes Rendus de l'Acad\'{e}mie des Sciences S\'{e}rie~I, 1981, 293: 549-552.

\bibitem{fk1990} Falconer K. \newblock Fractal geometry: mathematical foundations and applications. New York: John Wiley \& Sons, 1990.

\bibitem{wen00} Wen Z Y. \newblock Fractal geometry-mathematical foundation. Shanghai: Shanghai Science and Technology Education Press, 2000.

\bibitem{wjm93} Wu J M. \newblock Null sets for doubling and dyadic doubling measures. Annales Academiae Scientiarum Fennicae Series A. I. Mathematica, 1993, 18(1): 77-91.

\end{thebibliography}
\end{document}